\documentclass{amsart}
\usepackage[utf8]{inputenc}
\usepackage{amsthm}
\usepackage{graphicx}
\usepackage{enumerate}
\usepackage{tikz}
\usepackage{pgfplots}
\usepackage{mathrsfs}
\usepackage{braket}

\newtheorem{theorem}{Theorem}[section]
\newtheorem*{theorem*}{Theorem}
\newtheorem{lemma}[theorem]{Lemma}
\newtheorem*{lemma*}{Lemma}
\newtheorem{proposition}[theorem]{Proposition}

\newtheorem{problem}[theorem]{Problem}
\newtheorem{corollary}[theorem]{Corollary}
\theoremstyle{definition} \newtheorem{example}[theorem]{Example}
\theoremstyle{definition} \newtheorem{definition}[theorem]{Definition}
\theoremstyle{definition} 

\newcommand{\CC}{\mathbb{C}}
\newcommand{\ZZ}{\mathbb{Z}}
\newcommand{\QQ}{\mathbb{Q}}
\newcommand{\RR}{\mathbb{R}}

\newcommand{\GL}{\operatorname{GL}}

\newcommand{\PSU}{\operatorname{PSU}}
\newcommand{\U}{\operatorname{U}}
\newcommand{\SU}{\operatorname{SU}}
\newcommand{\SO}{\operatorname{SO}}

\newcommand{\Sp}{\operatorname{Sp}}
\newcommand{\im}{\operatorname{im}}

\newcommand{\su}{\mathfrak{su}}
\renewcommand{\u}{\mathfrak{u}}
\newcommand{\so}{\mathfrak{so}}
\renewcommand{\deg}{\operatorname{deg}}
\newcommand{\tr}{\operatorname{tr}}
\newcommand{\Ad}{\operatorname{Ad}}

\newcommand{\affW}{\widetilde{W}}
\newcommand{\extW}{\widetilde{W}^{\mathrm{ext}}}
\DeclareMathOperator{\cdes}{cdes}
\DeclareMathOperator{\cDes}{cDes}
\DeclareMathOperator{\sgn}{sgn}
\DeclareMathOperator{\diag}{diag}
\DeclareMathOperator{\ad}{ad}
\newcommand{\B}{\mathscr{B}}
\newcommand{\bigP}{\mathscr{P}}

\newcommand{\sectionsymbol}{§}
\renewcommand{\S}{\operatorname{S}}

\renewcommand{\epsilon}{\varepsilon}
\newcommand{\p}{\mathfrak{p}}
\renewcommand{\P}{\mathcal{P}}
\newcommand{\G}{\mathcal{G}}
\newcommand{\g}{\mathfrak{g}}
\renewcommand{\H}{\mathcal{H}}
\newcommand{\h}{\mathfrak{h}}
\newcommand{\K}{\mathcal{K}}
\renewcommand{\k}{\mathfrak{k}}
\newcommand{\A}{\mathcal{A}}
\renewcommand{\a}{\mathfrak{a}}
\newcommand{\Q}{\mathcal{Q}}

\newcommand{\alc}{\mathscr{A}}
\newcommand{\cowt}{\hat{L}}

\begin{document}

\title{Large products of double cosets for symmetric subgroups}
\author{Brendan Pawlowski}
\date{}

\begin{abstract} 
    We consider the problem of classifying pairs $x,y \in \G$ such that $\K x \K y \K = \G$ where $\G$ is a simple compact connected Lie group and $\K$ is a symmetric subgroup. We give a necessary condition on $x,y$ for all simply connected $\G$, and a complete classification when $\G = \SU(n)$ and any symmetric $\K \subseteq \G$ \emph{except} the type AIII case $\K \simeq \S(\U(p) \times \U(n-p))$ with $p \neq n/2$. We also present some applications of these results to gate decompositions in quantum computing.
\end{abstract}

\maketitle

\section{Introduction}

Let $\G$ be a simple compact connected Lie group and $\theta : \G \to \G$ a group automorphism satisfying $\theta^2 = \operatorname{id}$. Let $\K$ be the fixed-point subgroup $\G^{\theta} = \{g \in \G : \theta(g) = g\}$, or more generally any union of connected components of $\G^{\theta}$. Subgroups $\K$ that can be obtained this way for some $\theta$ are called \emph{symmetric subgroups}. Here is the main problem considered in this paper.
\begin{problem} \label{prob:main}
Describe all pairs $x,y \in \G$ such that $\K x \K y \K = \G$.
\end{problem}
By $\K x \K y \K$ we mean the set $\{k_1 x k_2 y k_3 : k_1,k_2,k_3 \in \K\}$. Note that this set only depends on the double cosets $\K x \K$ and $\K y \K$.  There is a well-developed theory of these double cosets when $\K$ is a symmetric subgroup, which puts the double cosets in bijection with points of a certain convex polytope. Solutions to Problem~\ref{prob:main} will therefore be described in terms of this polytope.

Our main results are (1) a partial solution to Problem~\ref{prob:main} for simply connected $\G$, and (2) a complete solution to Problem~\ref{prob:main} for $\G = \SU(n)$ and most symmetric subgroups $\K$.
\begin{theorem} \label{thm:main-1}
Suppose $\G$ is simply connected and $\K x \K y \K = \G$. Fix a fundamental alcove $\alc$ for $(\g,\k)$. Then $\K y \K = \K x^{-1} \K$, and if $x = \exp(i\pi X)$ for $X \in \overline{\alc}$, then $X$ is fixed by every extended affine Weyl group element $f$ with $f(\alc) = \alc$.
\end{theorem}
The terms used in Theorem~\ref{thm:main-1} will be defined later. The concrete takeaway is that the double coset $\K x\K$ may be identified with a point $X$ in the polytope $\overline{\alc}$, and Theorem~\ref{thm:main-1} says $X$ must be fixed by a certain group of symmetries of $\overline{\alc}$.

Part (2) is easier to describe precisely. First, it can be shown that if $\G = \SU(n)$, then the following three explicit examples of $\theta$ account for all possibilities up to conjugation \cite[Ch. X, Table V]{helgason}. 
\begin{description}
\item[Type AI] $\theta(g) = \overline{g}$, so $\K$ is the special orthogonal group $\SO(n)$.
\item[Type AII] $n$ even and $\theta(g) = \Omega \overline{g} \Omega^{-1}$ where $\Omega = \left[ \begin{smallmatrix} 0 & -I_{n/2} \\ I_{n/2} & 0 \end{smallmatrix} \right]$, so $\K$ is the compact symplectic group
\begin{equation*}
\Sp(n/2) = \{g \in \SU(n) : \Omega \overline{g} \Omega^{-1} = g\}.
\end{equation*}
\item[Type AIII] $\theta(g) = J_p g J_p^{-1}$ where $J_p = \diag(\overbrace{1, \ldots, 1}^p, \overbrace{-1, \ldots, -1}^{n-p})$, so $\K$ is the subgroup of block-diagonal matrices
\begin{equation*}
\S(\U(p) \times \U(n-p)) = \left \{ \begin{bmatrix} V & 0 \\ 0 & W \end{bmatrix} : V \in \U(p), W \in \U(n-p), \det(V)\det(W) = 1 \right\}.
\end{equation*}
\end{description}

\begin{theorem} \label{thm:main-2}
    Suppose $\G = \SU(n)$ and $U,V \in \G$. Then $\K U \K V \K = \G$ if and only if the appropriate conditions below hold for the given involution $\theta$.
    \begin{description}
    \item[Type AI] $U\theta(U)^{-1}$ and $V\theta(V)^{-1}$ both have characteristic polynomial $x^n + (-1)^n$, or equivalently both have eigenvalues $e^{i\pi (n-2j+1)/n}$ for $j = 1, \ldots, n$.
    \item[Type AII] $U\theta(U)^{-1}$ and $V\theta(V)^{-1}$ both have characteristic polynomial $(x^{n/2} + (-1)^{n/2})^2$, or equivalently both have eigenvalues $e^{i\pi (n-4j+2)/n}$ for $j = 1, \ldots, n/2$, each with multiplicity 2.
    \item[Type AIII, $p=n/2$] $U\theta(U)^{-1}$ and $V\theta(V)^{-1}$ have the same eigenvalues, which are of the form $e^{\pm i\pi t_1}, \ldots, e^{\pm i\pi t_{n/2}}$ where $\tfrac{1}{2} \geq t_1 \geq \cdots \geq t_{n/2} \geq 0$ and $t_i + t_{n-i+1} = \tfrac{1}{2}$ for all $i$. In the specific case $\theta(g) = J_p g J_p^{-1}$ and $\K = \S(\U(p) \times \U(n-p))$, this is equivalent to requiring that the upper-left $p \times p$ corners of $U,V$ have the same singular values $\sigma_1 \geq \cdots \geq \sigma_p$ which satisfy $\sigma_i^2 + \sigma_{p-i+1}^2 = 1$ for all $i$.
    \end{description}
\end{theorem}
The conditions in Theorem~\ref{thm:main-2} turn out to be exactly equivalent to the necessary condition given by Theorem~\ref{thm:main-1}. However, in general Theorem~\ref{thm:main-1} is \emph{not} sufficient to guarantee $\G = \K x \K y \K$: see the example with $\G = \Sp(2)$ at the end of \sectionsymbol\ref{sec:SU-block}.

This work was motivated by gate decomposition problems in quantum computing. In classical computing, arbitrary Boolean functions $\{0,1\}^n \to \{0,1\}^n$ are built up from a small list of basic functions: NOT, AND, NAND, OR, etc.  Similarly, an arbitrary quantum operation on $n$ qubits is a unitary matrix $U \in \U(2^n)$, and we would like to write it as a product of unitaries of some special kinds that are easier to implement. For example, one might fix a small set of gates $S \subseteq \U(2^n)$ and ask to decompose a given $U \in \U(2^n)$ as a product of elements of $S$ plus \emph{single-qubit gates}, i.e.\ elements of $\U(2)^{\otimes n}$. If $S$ is the set of \emph{controlled-not} (CNOT) gates, this is known to be possible for arbitrary $U$ \cite{shende-markov-bullock}.

To give an explicit example, let $(\CC^2)^{\otimes n}$ have basis $\{\ket{b} : b \in \{0,1\}^n \text{ a binary word}\}$, ordered in lex order. Taking $n=2$, the CNOT gate with control qubit 1 and target qubit 2 is the unitary 
\begin{equation*}
C = \begin{bmatrix} 1 & 0 & 0 & 0 \\ 0 & 1 & 0 & 0 \\ 0 & 0 & 0 & 1 \\ 0 & 0 & 1 & 0 \end{bmatrix}, \qquad \text{so } C(\ket{b_1 b_2}) = \begin{cases}
\ket{b_1 b_2} & \text{if $b_1 = 0$}\\
\ket{b_1 \operatorname{NOT}(b_2)} & \text{if $b_1 = 1$}
\end{cases}
\end{equation*}
One can show \cite{2q-optimal} that any element of $\U(4)$ has the form $L_1 C L_2 C L_3 C L_4$ where $L_1, L_2, L_3, L_4 \in \U(2) \otimes \U(2)$; that is, $LCLCLCL = \U(4)$ where $L = \U(2) \otimes \U(2)$. On the other hand, the set $LCLCL$ is strictly smaller than $\U(4)$.

A shorter factorization is possible: in \cite{berkeley-gate} it is shown that the \emph{Berkeley gate}
\begin{equation*}
B = \begin{bmatrix}
        \cos(\pi/8) & 0 & 0 & i\sin(\pi/8)\\
        0  & \cos(3/\pi/8) & i\sin(3\pi/8) & 0\\
        0  & i\sin(3/\pi/8) & \cos(3\pi/8) & 0\\
        i\sin(\pi/8) & 0 & 0 & \cos(\pi/8)
\end{bmatrix}
\end{equation*}
satisfies $LBLBL = \U(4)$. As explained in \sectionsymbol\ref{sec:gates}, this is in fact an example of the type AI case of Theorem~\ref{thm:main-2}.

Part of the motivation for this work was a search for gates generalizing the Berkeley gate, and related potentially novel decompositions for quantum gates. In types AI and AII, Theorem~\ref{thm:main-2} does not give much to work with: there is a unique double coset $\K U\K$ such that $\K U\K U\K = \SU(n)$. However, in type AIII there are infinitely many double cosets with this property, and we will discuss how to recover the recent \emph{block ZXZ decomposition} circuit \cite{block-ZXZ} from Theorem~\ref{thm:main-2}.

We start by reviewing Cartan decompositions and other preliminaries in Section~\ref{sec:cartan}. In Section~\ref{sec:necc}, we prove Theorem~\ref{thm:main-1} and discuss other necessary conditions for $\G = \K x \K y \K$. Sections \ref{sec:SU-SO}--\ref{sec:SU-block} prove the three cases of Theorem~\ref{thm:main-2}. Finally, in Section~\ref{sec:gates} we discuss some applications to gate decompositions in quantum computing.

\section{Lie group preliminaries}
\label{sec:cartan}
In this section we review some material on Lie groups, especially the Cartan decomposition with respect to a symmetric subgroup. The following notation will be fixed for the rest of the paper:
\begin{itemize}
\item $\G$ a simple compact connected Lie group with Lie algebra $\mathfrak{g}$, which we assume is a subalgebra of $\u(n)$
\item $\theta : \G \to \G$ an involutive automorphism 
\item $\G^\theta = \{g \in \G : \theta(g) = g\}$ its fixed point subgroup 
\item $\H_0$ denotes the connected component of the identity in a subgroup $\H \subseteq \G$ 
\item $\K$ a symmetric subgroup, i.e.\ one satisfying $(\G^\theta)_0 \subseteq \K \subseteq \G^\theta$.
\item $\k$ the 1-eigenspace of the derivative $d\theta : \g \to \g$, and $\p$ the (-1)-eigenspace.
\item $\a$ a maximal abelian subalgebra of $\p$
\item $\h$ a maximal abelian subalgebra of $\g$ containing $\a$
\item $\A = \exp(\a)$ and $\P = \exp(\p)$
\item $\Ad_g : \g \to \g$ the derivative of the conjugation map $\G \to \G, x \mapsto gxg^{-1}$ for $g \in \G$
\item $\ad_X : \g \to \g, Y \mapsto [X,Y]$ for $X \in \g$.
\end{itemize}
We also note that $\exp(\k) = \K_0$, the Fraktur form of ``k'' being, regrettably, ``$\k$''.

\subsection{Cartan decomposition}
The decomposition $\g = \k \oplus \p$ is called a \emph{Cartan decomposition} of $\g$, and it lifts to a decomposition of $\G$ also called a Cartan decomposition.
\begin{theorem} \cite[Ch. V, Theorem 6.7]{helgason}
    \label{thm:cartan-facts} \hfill
\begin{enumerate}[(a)]
    \item $\P = \bigcup_{k \in \K} k\A k^{-1}$.
    \item $\G = \K\P = \K\A\K$.
\end{enumerate}
\end{theorem}
By ``the Cartan decomposition of $\G$'', we mean the expression $\G = \K\A\K$.

\begin{example} \label{ex:AI-KAK}
Say $\G = \SU(n)$ and $\theta(g) = \overline{g}$ and $\K = \SO(n)$. Then $\g = \su(n)$ is the set of $n \times n$ trace 0 skew-Hermitian matrices and $d\theta$ is again complex conjugation. Hence $\k = \so(n)$ is the subalgebra of real skew-symmetric matrices, and $\p$ the subspace of imaginary symmetric matrices. Then we can take $\a$ to be the subalgebra of imaginary diagonal matrices---this is a maximal abelian subalgebra of $\su(n)$, so of $\p$ as well.

$\P = \exp(\p)$ is now the set of unitary symmetric matrices, and Theorem~\ref{thm:cartan-facts}(a) says that any unitary symmetric matrix is diagonalizable by an orthogonal matrix. Part (b) says that any unitary matrix equals $OS$ where $O$ is real orthogonal and $S$ is unitary symmetric. An associated Cartan decomposition of $U \in \SU(n)$ is a factorization
\begin{equation*}
U = O_1 D O_2, \qquad O_1, O_2 \in \SO(n) \text{ and $D$ unitary diagonal}.
\end{equation*}
\end{example}

\begin{example}
Although we focus on the compact case here, Cartan decomposition does apply to more general Lie groups. For instance, if $\G = \GL(n,\CC)$ and $\theta(g) = (g^\dagger)^{-1}$, then the appropriate version of Theorem~\ref{thm:cartan-facts} yields 
\begin{itemize}
\item the polar decomposition of a matrix $M$: $M = UP$ where $U$ is unitary and $P$ is positive semidefinite.
\item the singular value decomposition of $M$: $M = UDV$ where $U,V$ are unitary and $D$ is real diagonal.
\end{itemize}
\end{example}

\subsection{Cartan doubles}
Theorem~\ref{thm:cartan-facts}(b) shows that any $\K$-double coset $\K x \K$ can be written as $\K x \K = \K a \K$ with $a \in \A$. To find $a$ from $x$ we use the \emph{Cartan double} $x\theta(x)^{-1}$.

\begin{theorem} \label{thm:cartan-double}
    If $\K x \K = \K y \K$, then $x\theta(x)^{-1}$ and $y\theta(y)^{-1}$ are conjugate by an element of $\K$. If $\G$ is simply connected, then the converse holds.
\end{theorem}
\begin{proof}
By Theorem~\ref{thm:cartan-facts} we can write $x = k_1 a k_2$ with $k_i \in \K, a \in \A$. Then
\begin{equation*}
x\theta(x)^{-1} = (k_1 a k_2)(k_1 a^{-1} k_2)^{-1} = k_1 a^2 k_1^{-1}.
\end{equation*}
If $\K y\K = \K x\K$ then we can write $y = k_3 a k_4$, so $y\theta(y)^{-1} = k_3 a^2 k_3^{-1}$ is $\K$-conjugate to $x\theta(x)^{-1}$. For the converse, see \cite[Ch. V, Theorem 6.7]{helgason}.
\end{proof}

The quantity $x\theta(x)^{-1}$ is sometimes called the \emph{Cartan double} of $x$. If $\G$ is simply connected, then Theorem~\ref{thm:cartan-double} says that $\K$-conjugacy classes of elements of $\A$ are in bijection with $\K$-double cosets. The theorem can fail if $\G$ is not simply connected. For instance, let $\G = \PSU(2)$ and $\K$ be the subgroup of diagonal matrices, the fixed-point subgroup of $\theta : g \mapsto \left[ \begin{smallmatrix} 1 & 0 \\ 0 & -1 \end{smallmatrix} \right]g\left[ \begin{smallmatrix} 1 & 0 \\ 0 & -1 \end{smallmatrix} \right]$. Then $x = \left[ \begin{smallmatrix} 0 & -1 \\ 1 & 0 \end{smallmatrix} \right]$ is obviously not in $\K e \K = \K$, but $x \theta(x)^{-1} = -e = e \theta(e)^{-1} = e$ in $\G$.  

\begin{example}
Continuing the case of $\G = \SU(n)$ and $\K = \SO(n)$ from Example~\ref{ex:AI-KAK}, the Cartan double of $U = O_1 D O_2$ is $U\overline{U}^{-1} = UU^T = O_1 D^2 O_1^T$. Thus we can diagonalize $UU^T$ to compute $D$ up to signs---and there is no loss of generality in assuming, say, every $D_{jj}$ has the form $e^{i\pi x}$ with $0 \leq x < 1$. To compute $O_1$, find a basis of \emph{real} orthogonal eigenvectors of $UU^T$. There are numerical issues to solve in implementing this, especially if $UU^T$ has repeated eigenvalues, but the Cartan decomposition $U = O_1 D O_2$ guarantees that it is possible. Then set $O_2 = O_1^T D^{-1} U$.
\end{example}

\subsection{Roots and fundamental alcoves}
We turn to the problem of nicely parameterizing the $\K$-double cosets, or equivalently the $\K$-conjugacy classes in $\A$ if $\G$ is simply connected. Let $\mathfrak{h}$ be a maximal abelian subalgebra of $\mathfrak{g}$ containing $\mathfrak{a}$. Write $\g^\CC = \g \otimes \CC$. The operators $\ad_H : \g^\CC \to \g^\CC, X \mapsto [H,X]$ commute for $H \in \mathfrak{h}$, and can be shown to be diagonalizable using the compactness of $\G$, so they are simultaneously diagonalizable. Hence $\g^\CC$ breaks up as a direct sum of eigenspaces $\bigoplus_{\alpha} (\g^\CC)_\alpha$, satisfying
\begin{equation*}
\ad_H(X) = [H,X] = \alpha(H)X \qquad \text{for all $H \in \h, X \in (\g^\CC)_\alpha$}.
\end{equation*}
Each eigenvalue $\alpha(H)$ depends linearly on $H$, i.e.\ $\alpha$ lies in the dual space $(\h^\CC)^*$. The zero eigenspace $(\g^\CC)_0$ is simply $\h^\CC$. The nonzero eigenvalues $\alpha$ such that $(\g^\CC)_{\alpha} \neq 0$ are called the \emph{roots} of $\g$ with respect to $\h$. Let $\Phi(\g)$ denote the set of roots (suppressing the dependence on $\h$).

\begin{definition}
The set $\Phi(\g,\k)$ of \emph{restricted roots} of $(\g,\k)$ (with respect to $\h$) is 
\begin{equation*}
\{\alpha \in \Phi(\g) : \alpha|_{\a} \neq 0\}.
\end{equation*}
The \emph{Stiefel diagram} $D(\g,\k)$ is the union of all hyperplanes $\{X \in \mathfrak{a} : \alpha(X) = n\}$ for some $\alpha \in \Phi(\g,\k)$ and $n \in \ZZ$. The connected components of the complement $\a \setminus D(k)$ are called \emph{alcoves}. A \emph{fundamental alcove} is one whose closure contains $0$.
\end{definition}
The notation $D(\g,\k)$ may seem underspecified since the Stiefel diagram depends on the particular choice of maximal abelian subalgebra $\a \subseteq \k$. However, any two choices of $\a$ are conjugate by $\K$ \cite[Ch. V, Lemma 6.3(ii)]{helgason}, so changing $\a$ only changes $D(\g,\k)$ by a linear isomorphism. The next theorem is fundamental for us: it shows how to parameterize spherical double cosets by points of a convex polyhedron.

\begin{theorem}[\cite{helgason}, Theorem 7.9(b), Ch. VII] \label{thm:fund-alcove}
    Let $\alc$ be a fundamental alcove for $(\g, \k)$. For $\G$ simply connected, the function $\overline{\alc} \to \K \backslash \G / \K$, $X \mapsto \K\exp(\pi i X)\K$ is a bijection.
\end{theorem}
Here $\K \backslash \G / \K$ denotes the set of $\K$-double cosets in $\G$. From now on we work in a fixed fundamental alcove $\alc(\g,\k)$.
\begin{definition}
Given $x \in \G$ (assumed simply connected), let $a(x)$ be the unique point of $\overline{\alc(\g,\k)}$ with $\K a(x) \K = \K x \K$.
\end{definition}

\begin{example} \label{ex:AI-alcove}
    Say $\G = \SU(n)$ and $\K = \SO(n)$ and $\a = \h$ consists of the imaginary diagonal matrices, as in Example~\ref{ex:AI-KAK}. The roots with respect to $\h$ are $\epsilon_i - \epsilon_j$ for $i \neq j$, where $\epsilon_i$ is the linear functional sending $D \in \mathfrak{h}^\CC$ to $D_{ii}$. The restricted roots are no different, since $\a = \h$. Identify $\{{\bf x} \in \RR^n : \sum_i x_i = 0\}$ with $\h$ by the correspondence ${\bf x} \mapsto \diag(ix_1, \ldots, ix_n)$. Then the Stiefel diagram is the union of the hyperplanes $\{x_i-x_j = n\}$ over all $n \in \ZZ$.

    For instance, if $n = 3$ and we identify $\{x_1 + x_2 + x_3 = 0\}$ isometrically with $\RR^2$:
    
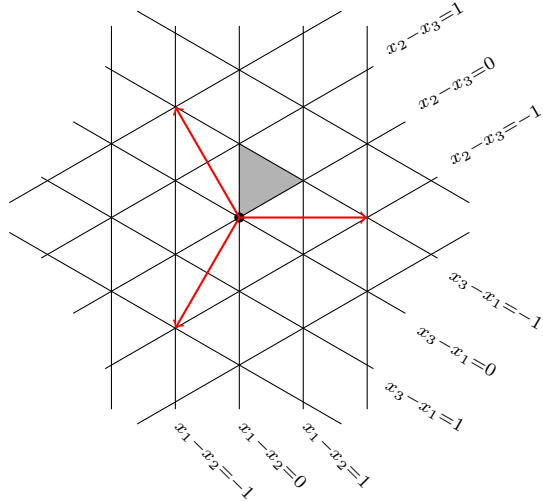
\begin{center}
        \begin{tikzpicture}[scale=0.85]
            \filldraw[color=black!30!white] (0,0) -- (1,0.577) -- (0,1.154) -- (0,0);
            \filldraw (0,0) circle [radius=2pt];
            
            \draw (-2,-3) -- (-2,3);
            \draw (-1,-3) -- (-1,3);
            \draw (0,-3) -- (0,3);
            \draw (1,-3) -- (1,3);
            \draw (2,-3) -- (2,3);

            \draw (-2.598+2*0.5, -1.5-2*0.866) -- (2.598+2*0.5, 1.5-2*.866);
            \draw (-2.598+0.5, -1.5-0.866) -- (2.598+0.5, 1.5-0.866);
            \draw (-2.598, -1.5) -- (2.598, 1.5);
            \draw (-2.598-0.5, -1.5+0.866) -- (2.598-0.5, 1.5+0.866);
            \draw (-2.598-2*0.5, -1.5+2*0.866) -- (2.598-2*0.5, 1.5+2*0.866);

            \draw (-2.598+2*0.5, 1.5+2*0.866) -- (2.598+2*0.5, -1.5+2*0.866);
            \draw (-2.598+0.5, 1.5+0.866) -- (2.598+0.5, -1.5+0.866);
            \draw (-2.598, 1.5) -- (2.598, -1.5);
            \draw (-2.598-0.5, 1.5-0.866) -- (2.598-0.5, -1.5-0.866);
            \draw (-2.598-2*0.5, 1.5-2*0.866) -- (2.598-2*0.5, -1.5-2*0.866);


            \node[above right] at (2.598+0.5, 1.5-0.866) {\rotatebox{30}{$\scriptstyle x_2-x_3 = -1$}};
            \node[above right] at (2.598, 1.5) {\rotatebox{30}{$\scriptstyle x_2-x_3 = 0$}};
            \node[above right] at (2.598-0.5, 1.5+0.866) {\rotatebox{30}{$\scriptstyle x_2-x_3 = 1$}};

            \node[below right] at (2.598, -1.5) {\rotatebox{-30}{$\scriptstyle x_3-x_1 = 0$}};
            \node[below right] at (2.598-0.5, -1.5-0.866) {\rotatebox{-30}{$\scriptstyle x_3-x_1 = 1$}};
            \node[below right] at (2.598+0.5, -1.5+0.866) {\rotatebox{-30}{$\scriptstyle x_3-x_1 = -1$}};

            \node[below right] at (-1-0.2,-3) {\rotatebox{-45}{$\scriptstyle x_1-x_2=-1$}};
            \node[below right] at (0-0.2,-3) {\rotatebox{-45}{$\scriptstyle x_1-x_2=0$}};
            \node[below right] at (1-0.2,-3) {\rotatebox{-45}{$\scriptstyle x_1-x_2=1$}};

            \draw[red, thick, ->] (0,0) -- (2,0);
            \draw[red, thick, ->] (0,0) -- (-1, 1.732);
            \draw[red, thick, ->] (0,0) -- (-1, -1.732);
        \end{tikzpicture}
    \end{center}
    The three vectors, counterclockwise from right, are $(1,-1,0), (0,1,-1), (-1,0,1)$, the normals to the hyperplanes. A choice of fundamental alcove $\alc$ has been shaded. In general, $\{\mathbf{x} \in \RR^n : x_1 > \cdots > x_n > x_1-1, \sum_j x_j = 0\}$ can be taken as a fundamental alcove.

    Theorem~\ref{thm:fund-alcove} applied to this case therefore says that any $U \in \SU(n)$ can be written as $O_1 D O_2$ with $O_1,O_2 \in \SO(n)$ and $D = \diag(e^{\pi i a_1}, \ldots, e^{\pi i a_n})$ for a \emph{unique} vector $(a_1, \ldots, a_n)$ satisfying $a_1 \geq \cdots \geq a_n \geq a_1-1$ and $\sum_i a_i = 0$. 

\end{example}

\begin{example} \label{ex:GxG}
    We can always take $\G = \G' \times \G'$ where $\G'$ is a compact connected Lie group, and set $\theta(x,y) = (y,x)$. Then
    \begin{itemize}
        \item $\K$ is the diagonal subgroup $\{(x,x) : x \in \G'\}$ and $\mathfrak{p} = \{(X,-X) : X \in \mathfrak{g}'\}$.
        \item $\phi : \G/\K \to \G'$, $(x,y)\K \mapsto xy^{-1}$ is a diffeomorphism sending a left $\K$-orbit $\K p$ onto the conjugacy class of $\phi(p)$. Thus  $\K$-double cosets are equivalent to conjugacy classes in $\G'$.
        \item We can take $\mathfrak{a} = \{(X,-X) : X \in \mathfrak{h}'\}$ where $\mathfrak{h}'$ is a maximal abelian subalgebra of $\mathfrak{g}'$, and $\mathfrak{h} = \mathfrak{h}' \oplus \mathfrak{h}'$.
    \item The roots of $\mathfrak{g}$ with respect to $\mathfrak{h}$ are the functionals $\alpha \oplus 0$ and $0 \oplus \alpha$ where $\alpha$ ranges over the roots of $\mathfrak{g}'$ with respect to $\mathfrak{h}'$.
       \item The Stiefel diagram $D(\g,\k)$ is the union of the hyperplanes $\{(X,-X) : X \in \mathfrak{h}', \gamma(X,-X) = n\}$ over roots $\gamma$ of $\g$ and $n \in \ZZ$ as above.
    \end{itemize} 


    Forgetting about symmetric subgroups for the moment, define the \emph{Stiefel diagram} $D(\G')$ of $\G'$ to be the union of all hyperplanes $\{X \in \mathfrak{h}' : \alpha(X) = n\}$ over roots $\alpha$ and $n \in \ZZ$. A fundamental alcove $\alc(\G')$ is a connected component of $\h' \setminus D(\G')$ whose closure contains $0$.

    Now, $d\phi$ identifies the tangent space $(G/K)_{eK} = \mathfrak{p}$ with $\mathfrak{g}'$ and sends $(X,-X)$ to $2X$. Hence it identifies $\alc(\G')$ with $2\alc(\g, \k)$. Applying Theorem~\ref{thm:fund-alcove} then shows that for $\G'$ simply connected, each element of $\G'$ is conjugate to $\exp(2\pi i X)$ for a unique $X \in \overline{\alc(\G')}$.

    In the case $\G' = \SU(n)$, this just says a unitary matrix is unitarily similar to a diagonal matrix $\diag(e^{2i\pi x_1}, \ldots, e^{2i\pi x_n})$ for a unique $(x_1, \ldots, x_n) \in \RR^n$ with $x_1 \geq \cdots \geq x_n \geq x_1-1$. Note the extra factor of 2 compared to the conclusion of Example~\ref{ex:AI-alcove}.
\end{example}

\begin{lemma} \label{lem:KAK-conj}
    For $\G$ simply connected, $\K x \K = \K y \K$ if and only if $x\theta(x)^{-1}$ and $y\theta(y)^{-1}$ are conjugate by $\G$.
\end{lemma}

\begin{proof}
    Recall that
    \begin{align*}
        &D(\g) = \bigcup_{n \in \ZZ, \alpha} \{X \in \mathfrak{h} : \alpha(X) = n\} \qquad \text{($\alpha$ a root of $\mathfrak{g}$)} \\
        &D(\g,\k) = \bigcup_{n \in \ZZ, \alpha} \{X \in \mathfrak{a} : \alpha(X) = n\} \qquad \text{($\alpha$ a restricted root of $(\g,\k)$)}
    \end{align*}
    are the Stiefel diagrams of $\g$ (cf. Example~\ref{ex:GxG}) and of $(\g,\k)$ respectively.

    The forward direction of the proposition is immediate from Theorem~\ref{thm:cartan-double}. Conversely, suppose $x\theta(x)^{-1}$ and $y\theta(y)^{-1}$ are conjugate in $\G$. By Theorem~\ref{thm:cartan-facts} and Theorem~\ref{thm:fund-alcove} we can write $x = k_1 \exp(\pi i X)k_2$, $y = k_3\exp(\pi i Y)k_4$ with $k_i \in \K$ and $X,Y$ both in a fixed (closed) alcove, i.e. the closure of a connected component of $\a \setminus D(\g,\k)$. This means $X,Y$ are also in the same connected component (closure) of $\g \setminus D(\g)$, since $D(\g,\k) = D(\g) \cap \mathfrak{\a}$. By assumption $x\theta(x)^{-1} = \exp(2\pi i X)$ and $y\theta(y)^{-1} = \exp(2\pi i Y)$ are conjugate in $\G$, so by the conclusion of Example~\ref{ex:GxG} we must have $X = Y$.
\end{proof}

\subsection{Weyl groups}
The \emph{affine Weyl group} $\affW(\g,\k)$ is the group of affine transformations of $\mathfrak{a}$ generated by all reflections across the hyperplanes defining $D(\g,\k)$, i.e. the hyperplanes
\begin{equation*}
H_{\alpha,n} := \{X \in \a : \alpha(X) = n\} \qquad \text{for $\alpha \in \Phi(\g,\k), n \in \ZZ$}.
\end{equation*}
Let $s_{\alpha,n}$ denote the reflection across $H_{\alpha,n}$, so $s_{\alpha,n}(v) = v - ((v,\alpha)-n)\alpha^\vee$. The \emph{(finite) Weyl group} $W(\g,\k) \subseteq \affW(\g,\k)$ is the subgroup of linear transformations generated by all reflections $s_{\alpha,0}$.

It is clear from this definition that $\affW$ permutes the set of alcoves, and in fact this action is simply transitive \cite[Ch. VII, Corollary 7.4]{helgason}: given two alcoves $\alc_1, \alc_2$, there is a \emph{unique} $f \in \affW$ with $f(\alc_1) = \alc_2$.

We shall need two lattices closely related to the Stiefel diagram. First, the \emph{coroot} associated to a root $\alpha \in \Phi(\g,\k)$ is $\alpha^\vee = \tfrac{2}{(\alpha,\alpha)}\alpha$, and the \emph{coroot lattice} is the lattice $L(\Phi(\g,\k)^\vee)$ generated by the coroots. Let $\tau_X : \a \to \a$ denote translation by an element $X \in \a$.
\begin{lemma} \label{lem:coroot-translate}
If $X$ is an element of the coroot lattice $L(\Phi(\g,\k))^\vee$, then $\tau_X \in \affW(\g,\k)$.
\end{lemma}
\begin{proof}
    It suffices to prove the claim assuming $X = \alpha^\vee$ is a coroot. In this case, a quick calculation using the fact that $(\alpha,\alpha^\vee) = 2$ shows that $s_{\alpha,2} s_{\alpha,1} = \tau_{\alpha^\vee}$.
\end{proof}

The next lemma shows that, at least in the simply connected case, the affine Weyl group $\affW$ exactly captures the indeterminacy in choosing an $A \in \mathfrak{a}$ such that $\K \exp(A) \K$ equals a fixed $\K$-double coset.
\begin{lemma} \label{lem:affW-orbit}
    Suppose $\G$ is simply connected and $X,Y \in \mathfrak{a}$. Then $\K\exp(\pi i X)\K = \K\exp(\pi i Y)\K$ if and only if $X, Y$ are in the same orbit of $\affW(\g,\k)$.
\end{lemma}

\begin{proof}
    The affine Weyl group $\affW$ is generated by the finite Weyl group $W$ together with the subgroup of translations $\tau_H$ for $H$ in the coroot lattice $L(\Phi(G,K)^\vee)$ \cite[Ch. VII, Lemma 7.1]{helgason}. Consider these two subgroups separately:
    \begin{enumerate}[(a)]
    \item If $H \in L(\Phi^\vee)$, then $\exp(2\pi i H) = e$ \cite[Ch. VII, Lemma 7.6]{helgason}.
    \item Viewed as a group of linear transformations of $\a$, the quotient group
    \begin{equation*}
    \frac{\{k \in \K : \Ad_k(\a) \subseteq \a\}}{ \{k \in \K : \Ad_k(A) = A \text{ for all $A \in \a$}\} }
    \end{equation*}
    is the same as $W$ \cite[Ch. VII, \sectionsymbol 2]{helgason}. In particular, if $w \in W$ then $\exp(A)$ and $\exp(w(A))$ are $\K$-conjugate.
    \end{enumerate}
    Now take $f \in \affW$ and set $x = \exp(i\pi X)$ and $x_f = \exp(i\pi f(X))$, so $x\theta(x)^{-1} = \exp(2i\pi X)$ and $x_f\theta(x_f)^{-1} = \exp(2i\pi f(X))$. If $f$ has the form $\tau_H$, then these are equal by (a). If $f \in W$, then they are $\K$-conjugate by (b). By Lemma~\ref{lem:KAK-conj}, $\K x \K = \K x_f \K$ holds for either type of $f$, and hence for all $f \in \affW$.

Conversely, suppose $\K\exp(\pi i X)\K = \K\exp(\pi i Y)\K$. Let $f_X, f_Y$ be the unique elements of $\affW(G,K)$ with $f_X(X), f_Y(Y) \in \overline{\alc}$. By the previous paragraph we have $\K\exp(\pi i f_X(X))\K = \K \exp(\pi i X) \K$, and likewise for $Y$. But then $\K\exp(\pi i f_X(X))\K = \K\exp(\pi i f_Y(Y))\K$ and hence $f_X(X) = f_Y(Y)$ by Lemma~\ref{lem:KAK-conj}. 
\end{proof}

The second lattice we need is the \emph{coweight lattice}
\begin{equation*}
\cowt(\Phi(\g,\k)^\vee) = \{X \in \a : \alpha(X) \in \ZZ \text{ for all $\alpha \in \Phi(\g,\k)$}\}.
\end{equation*}
Note that these are exactly the points in the Stiefel diagram where the largest possible number of hyperplanes intersect. It is a basic fact about root systems that $(\alpha^\vee, \beta) \in \ZZ$ for any roots $\alpha,\beta$, so $\cowt(\Phi(\g,\k)^\vee)$ contains the coroot lattice $L(\Phi(\g,\k)^\vee)$. They need not be equal.

Translation by a coweight $X$ maps the hyperplane $H_{\alpha,n}$ to another hyperplane $H_{\alpha,n+\alpha(X)}$, hence preserves the Stiefel diagram and maps alcoves to alcoves. However, these translations do \emph{not} necessarily lie in $\affW$. This suggests the next definition.
\begin{definition}
Fix a fundamental alcove $\A$ for $(\g,\k)$. Given a coweight $X \in \cowt(\Phi(\g,\k)^\vee)$, let $f_X$ be the unique element of $\affW(\g,\k)$ satisfying $f_X(\A) = \A + X$.
\end{definition}

\begin{definition}
The \emph{extended affine Weyl group} $\extW(\g,\k)$ is the group of affine transformations of $\mathfrak{a}$ generated by $\affW$ and translations by $\hat{L}(\Phi(G,K)^{\vee})$.
\end{definition} 

The group appearing in Theorem~\ref{thm:main-1} is $\extW(\g,\k)_{\alc}$, the (setwise) stabilizer  of $\alc$ in $\extW(\g,\k)$. This group measures the difference between the coweight and coroot lattices. 
\begin{proposition} \cite[\sectionsymbol 5]{lam-postnikov} \label{prop:extW-stab} Sending $X \mapsto \tau_{-X} f_{X}$ defines an isomorphism
    \begin{equation*}
    \hat{L}(\Phi(G,K)^{\vee}) / L(\Phi(G,K)^{\vee}) \to \extW_{\alc}.
    \end{equation*}
\end{proposition} 

\begin{example} \label{ex:lattices}
    Consider again the case $\G = \SU(3), \K = \SO(3)$ from Example~\ref{ex:AI-alcove}. The (restricted) roots are $\{\pm \alpha_0, \pm \alpha_1, \pm \alpha_2\}$ where
    \begin{equation*}
    \alpha_1 = \epsilon_1 - \epsilon_2, \quad \alpha_2 = \epsilon_2-\epsilon_3, \quad \alpha_0 = \epsilon_3-\epsilon_1.
    \end{equation*}
    These are the same as the corresponding coroots.

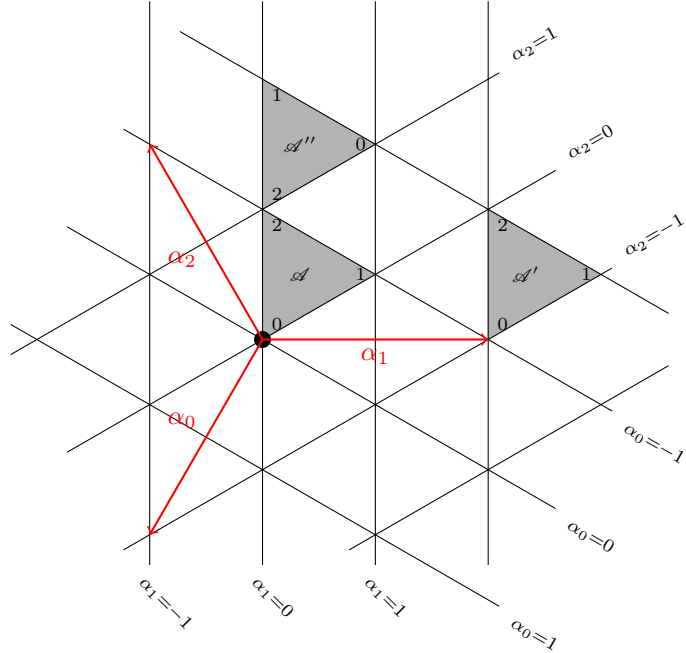
\begin{center}
        \begin{tikzpicture}[scale=1.5]
            \filldraw[color=black!30!white] (0,0) -- (1,0.577) -- (0,1.154) -- (0,0);
            \filldraw[color=black!30!white] (2,0) -- (3,0.577) -- (2,1.154) -- (2,0);
            \filldraw[color=black!30!white] (0,0+1.154) -- (1,0.577+1.154) -- (0,1.154+1.154) -- (0,0+1.154);
            \filldraw (0,0) circle [radius=2pt];
            \node at (0.333, 0.577) {$\scriptstyle \alc$};
            \node at (2.333, 0.577) {$\scriptstyle \alc'$};
            \node at (0.333, 0.577+1.154) {$\scriptstyle \alc''$};

            \node[above right] at (0,0) {$\scriptstyle 0$};
            \node[left] at (1,0.577) {$\scriptstyle 1$};
            \node[below right] at (0,1.154) {$\scriptstyle 2$};

            \node[above right] at (2,0) {$\scriptstyle 0$};
            \node[left] at (3,0.577) {$\scriptstyle 1$};
            \node[below right] at (2,1.154) {$\scriptstyle 2$};

            \node[above right] at (0,1.154) {$\scriptstyle 2$};
            \node[left] at (1,0.577+1.154) {$\scriptstyle 0$};
            \node[below right] at (0,1.154+1.154) {$\scriptstyle 1$};
            
            \draw (-1,-2) -- (-1,3);
            \draw (0,-2) -- (0,3);
            \draw (1,-2) -- (1,3);
            \draw (2,-2) -- (2,3);

            \draw (-1.732+0.5+2, -1-0.866) -- (2.598+2*0.5, 1.5-2*.866); 
            \draw (-1.732+0.5, -1-0.866) -- (2.598+0.5, 1.5-0.866);
            \draw (-1.732, -1) -- (2.598, 1.5);
            \draw (-1.732-0.5, -1+0.866) -- (2.598-0.5, 1.5+0.866);

            \draw (-1.732+2*0.5, 1+2*0.866) -- (2.598+2*0.5, -1.5+2*0.866); 
            \draw (-1.732+0.5, 1+0.866) -- (2.598+0.5, -1.5+0.866);
            \draw (-1.732, 1) -- (2.598, -1.5);
            \draw (-1.732-0.5, 1-0.866) -- (2.598-0.5, -1.5-0.866);


            \node[above right] at (2.598+0.5, 1.5-0.866) {\rotatebox{30}{$\scriptstyle \alpha_2 = -1$}};
            \node[above right] at (2.598, 1.5) {\rotatebox{30}{$\scriptstyle \alpha_2 = 0$}};
            \node[above right] at (2.598-0.5, 1.5+0.866) {\rotatebox{30}{$\scriptstyle \alpha_2 = 1$}};

            \node[below right] at (2.598, -1.5) {\rotatebox{-30}{$\scriptstyle \alpha_0 = 0$}};
            \node[below right] at (2.598-0.5, -1.5-0.866) {\rotatebox{-30}{$\scriptstyle \alpha_0 = 1$}};
            \node[below right] at (2.598+0.5, -1.5+0.866) {\rotatebox{-30}{$\scriptstyle \alpha_0 = -1$}};

            \node[below right] at (-1-0.2,-2) {\rotatebox{-45}{$\scriptstyle \alpha_1=-1$}};
            \node[below right] at (0-0.2,-2) {\rotatebox{-45}{$\scriptstyle \alpha_1=0$}};
            \node[below right] at (1-0.2,-2) {\rotatebox{-45}{$\scriptstyle \alpha_1=1$}};

            \draw[red, thick, ->] (0,0) -- (2,0);
            \draw[red, thick, ->] (0,0) -- (-1, 1.732);
            \draw[red, thick, ->] (0,0) -- (-1, -1.732);
            \node[red, below] at (1,0) {$\alpha_1$};
            \node[red, below left] at (-0.5, 0.866) {$\alpha_2$};
            \node[red, above left] at (-0.5, -0.866) {$\alpha_0$};
        \end{tikzpicture}
    \end{center}
    Translation by $\alpha_1$ maps the fundamental alcove $\alc$ to $\alc'$, and one can see $\tau_{\alpha_1} = s_{\alpha_1,2}s_{\alpha_1,1} \in \affW$. 

    Let $\omega_0, \omega_1, \omega_2$ be the vertices of $\alc$ labeled by $0,1,2$. The coweight lattice, generated by $\omega_1, \omega_2$, consists of the points where hyperplanes intersect. Translation by $\omega_2$ maps $\alc$ to $\alc''$, but $\tau_{\omega_2}$ is \emph{not} an element of $\affW$. Instead, the element $f_{\omega_2} \in \affW$ sending $\alc$ to $\alc''$ is $s_{\alpha_2,1}s_{\alpha_0,-1}$, which one can see is not a translation by considering how the vertices $0,1,2$ are moved. The element $\tau_{-\omega_2}f_{\omega_2} \in \extW_\alc$ is the $60^\circ$ rotation of $\alc$ mapping $0,1,2$ to $2,0,1$.
\end{example}

We will use an explicit realization of $\extW_{\alc}$ due to Lam and Postnikov \cite{lam-postnikov}.  Choose simple roots $\alpha_1, \ldots, \alpha_r$ for $\Phi(\g,\k)^\vee$, and write $\alpha > 0$ or $\alpha < 0$ to indicate that a root is positive or negative with respect to this simple system. Define $\alpha_0$ by letting $-\alpha_0$ be the \emph{highest root}, i.e.\ the unique root such that $-\alpha_0 - \beta \geq 0$ for all positive roots $\beta$. Define integers
\begin{equation*}
    a_i = \begin{cases}
        -\alpha_0(\omega_i) & \text{for $i = 1, \ldots, r$}\\
        1 & \text{for $i = 0$},
    \end{cases}
\end{equation*}
Here, $\omega_1, \ldots, \omega_r$ are the \emph{fundamental coweights}: a dual basis to the simple roots $\alpha_1, \ldots, \alpha_r$ under the inner product $(-,-)$. By convention, $\omega_0 = 0$.
\begin{proposition}
The points $a_i^{-1}\omega_i$ for $i = 0, 1, \ldots, r$ are the vertices of $\overline{\alc}$.
\end{proposition}

\begin{definition}
Call $i \in \{0,1,\ldots,r\}$ a \emph{cyclic descent} of $w \in W(\g,\k)$ if $w(\alpha_i) < 0$. Let $\cDes(w)$ be the set of cyclic descents of $w$, and define a coweight 
\begin{equation*}
    \delta_w = \sum_{i \in \cDes(w)} \omega_i
\end{equation*}
and a statistic
\begin{equation*}
    \cdes(w) = \sum_{i \in \cDes(w)} a_i.
\end{equation*}
\end{definition}

\begin{theorem}[\cite{lam-postnikov}, Proposition 6.4] \label{thm:extW-stab} Let $C(\g,\k) = \{w \in W(\g,\k) : \cdes(w) = 1\}$. Then $C$ is the subgroup $W \cap \{f_{X} : X \in \cowt(\Phi(\g,\k)^\vee)\}$, and $w \mapsto w\tau_{-\delta_w}$ is an isomorphism $C \to \extW_{\alc}$.
\end{theorem}

\begin{example}
If $\G = \SU(n)$ and $\K = \SO(n)$, we have $W = S_n$. A permutation $w = w_1 \cdots w_n$ has a cyclic descent at $i > 0$ if $w_i > w_{i+1}$, and a cyclic descent at $0$ if $w_n > w_1$. The permutations with exactly one cyclic descent are
\begin{equation*}
i(i+1) \cdots n 12\cdots (i-1) \quad \text{for $i = 1, \ldots, n$},
\end{equation*}
so $C \simeq \ZZ/n\ZZ$ is the cyclic subgroup generated by the long cycle $(1,2,\ldots,n)$. For example, the $60^\circ$ rotation of $\alc$ found in Example~\ref{ex:lattices} in fact generates $\extW_\alc$ in the case $n=3$.
\end{example}

\subsection{Basics on quantum Littlewood-Richardson coefficients}
This subsection is independent from the previous, and will only be used as technical background for \sectionsymbol\ref{subsec:P-AI}. Let $[n] = \{1, 2, \ldots, n\}$, and write ${[n] \choose k}$ for the set of $k$-subsets of $[n]$. For each triple $I,J,K \in {[n] \choose k}$ and integer $d \geq 0$, there is an associated \emph{quantum Littlewood-Richardson coefficient} $c_{IJ}^{K,d}$. These numbers arise as certain cohomological invariants of Grassmannians \cite{buch-QH}, as well as irreducible multiplicities in some $\GL(n)$-representations \cite{pawlowski-positroid-classes}. More relevantly here, they also appear in solving the \emph{multiplicative eigenvalue problem}: as $U_1, U_2$ range over all unitary matrices with fixed spectra $\Lambda_1, \Lambda_2$, what are the possible spectra $\Lambda_{12}$ of $U_1 U_2$ in terms of $\Lambda_1, \Lambda_2$? Agnihotri and Woodward solved this problem by showing that the possible $\Lambda_{12}$ are characterized by linear inequalities defined by quantum Littlewood-Richardson coefficients \cite{agnihotri-woodward}.

Giving a full definition of the coefficients $c_{IJ}^{K,d}$ would be rather involved, but fortunately we only require one simple combinatorial property they satisfy, for which the following sketch will suffice. Let $q$ be an indeterminate. For $0 < k < n$, the \emph{small quantum cohomology ring} of the Grassmannian of $k$-planes in $\CC^n$ is a $\ZZ[q]$-algebra $\operatorname{QH}_{k,n}$. It is free of rank ${n \choose k}$, with a distinguished basis $\{\sigma_I : I \in {[n] \choose k}\}$. The quantum Littlewood-Richardson coefficients express the structure constants of this basis:
\begin{equation} \label{eq:qLR}
\sigma_I \sigma_J = \sum_{d \geq 0, K \in {[n] \choose k}} c_{IJ}^{K,d} q^d \sigma_K.
\end{equation}

The only further fact we will need is that this ring is graded, with degrees 
\begin{equation} \label{eq:QH-grading}
\deg(\sigma_I) = \sum_{j=1}^{k} (n-k+j-I_j) = k(n-k) + {k+1 \choose 2} - \sum I, \qquad \deg(q) = n.
\end{equation}
This grading may look strange. A clearer picture emerges using the fact that ${[n] \choose k}$ is in bijection with the set of Young diagrams $\lambda$ contained in a $k \times (n-k)$ grid: in this indexing, the degree of $\sigma_{\lambda}$ is just the number of boxes in $\lambda$. However, \eqref{eq:QH-grading} will suffice for us. 

\begin{proposition} \label{prop:grading}
If $c_{IJ}^{K,d} \neq 0$, then $\sum I + \sum J - \sum K = k(n-k) + {k+1 \choose 2} - nd$.
\end{proposition}
\begin{proof}
Set $D = k(n-k) + {k+1 \choose 2}$. Then the only nonzero terms in \eqref{eq:qLR} occur when 
\begin{align*}
&\deg(\sigma_I) + \deg(\sigma_J) = \deg(q^d) + \deg(\sigma_K)\\
\Rightarrow\, &D - \sum I + D - \sum J = nd + D - \sum K.
\end{align*}
\end{proof}

\section{Necessary conditions for $\G = \K x \K y \K$} \label{sec:necc}
In this section we prove some necessary conditions on pairs $x,y \in \G$ with $\G = \K x \K y \K$, assuming $\G$ simply connected. First we reduce to a Lie algebra problem.
\begin{definition}
Let $\B(\G,\K) = \{X \in \overline{\alc(\g,\k)} : \K\exp(\pi i X)\K\exp(-\pi i X)\K = \G\}$.
\end{definition}
The next proposition shows that $\B(\G,\K)$ completely describes the pairs $x,y$ with $\K x \K y \K = \G$.
\begin{proposition}
    Assume $\G$ is simply connected. Then $\K x \K y \K = \G$ if and only if $a(x) = a(y^{-1})$ and $a(x) \in \B(\G,\K)$. 
\end{proposition}
\begin{proof}
    If $\K x \K y \K = \G$, then certainly $e \in \K x \K y \K$, so $\K y \K = \K x^{-1} \K$. Therefore $a(y) = a(x^{-1})$ by Theorem~\ref{thm:fund-alcove}. Since $x = \exp(\pi i a(x))$ we see $a(x) \in \B(\G,\K)$.
\end{proof}

We can now state Theorem~\ref{thm:main-1} more precisely and prove it.
\begin{theorem*}[Theorem~\ref{thm:main-1}] 
    Suppose $\G$ is compact, simple, and simply connected, and $x = \exp(\pi i X) \in \B(\G,\K)$ where $X \in \overline{\alc}$. Then $f(X) = X$ for all $f \in \extW_{\alc}$.
\end{theorem*} 

\begin{proof}
    Take $f \in \extW_{\alc}$. By Proposition~\ref{prop:extW-stab}, $f = f_Z^{-1} \tau_Z$ for some $Z$ in the coweight lattice $\hat{L}(\Phi^{\vee})$. Set $z = \exp(2\pi i Z)$ and $\sqrt{z} = \exp(\pi i Z)$. Then $z$ is in the center $Z(G)$ \cite[Ch. VII, Lemma 6.5]{helgason}. Since $\K x \K x^{-1} \K = \G$ by assumption, we have $\sqrt{z} \in \K x \K x^{-1} \K$, i.e.\ $\sqrt{z}kx \in \K x \K$ for some $k \in \K$. Write $\sim$ for conjugacy in $\G$ and $x \,^\K\!\!\sim^\K y$ to mean $\K x \K = \K y \K$. As $x \,^\K\!\!\sim^\K \sqrt{z}kx$, Lemma~\ref{lem:KAK-conj} gives 
    \begin{align*}
        x^2 = x \theta(x)^{-1} &\sim \sqrt{z}kx \cdot \theta(\sqrt{z}kx)^{-1} = \sqrt{z} kx^2 k^{-1} \sqrt{z}\\
        &\sim kx^2 k^{-1} z = kx^2 z k^{-1} \sim x^2 z,
    \end{align*}
    i.e.\ $\exp(2\pi i X) \sim \exp(2\pi i (X+Z)) = \exp(2\pi i \tau_Z(X))$. But now 
    \begin{align*}
        x = \exp(\pi i X) &\,^\K\!\!\sim^\K \exp(\pi i \tau_Z(X)) \qquad \text{(by Lemma~\ref{lem:KAK-conj})}\\
        &\,^\K\!\!\sim^\K \exp(\pi i f_Z^{-1} \tau_Z(X)) = \exp(\pi i f(X)) \qquad \text{(by Lemma~\ref{lem:affW-orbit})}
    \end{align*} Since both $X$ and $f(X)$ are in the (closed) fundamental alcove $\overline{\alc}$, this forces $X = f(X)$ by Theorem~\ref{thm:fund-alcove}.
\end{proof}

 Following \cite{peterson-crooks-smith}, we now describe a different method for deriving linear \emph{inequalities} on $\B(\G,\K)$. The reader who is only interested in the specific $\G = \SU(n)$ results of Theorem 1.2 can skip this material, because Theorem~\ref{thm:main-1} will suffice. However, it seems likely that this method gives stronger results than Theorem~\ref{thm:main-1} for more general $\G$.

Recall from Example~\ref{ex:GxG} that any $g \in \G$ is conjugate to $\exp(2\pi i X)$ for a unique $X \in \overline{\alc(\G)}$. As before we write $\sim$ for conjugacy in $\G$. Let 
\begin{equation*}
    \bigP(\G) = \{(X_0,X_1,X_2) \in \overline{\alc(\G)}^3 : \text{$\exists x_0,x_1,x_2 \in \G$ with $x_0 = x_1 x_2$ and $x_j \sim \exp(2\pi i X_j)$}\}.
\end{equation*}
In words, $\bigP(\G)$ records the possible conjugacy classes of elements $x_0,x_1,x_2$ with $x_0 = x_1 x_2$.  Also define 
\begin{align*}
    \bigP(\G,\K) = \{(X_0,X_1,X_2) \in \overline{\alc(\G,\K)}^3 :\, &\text{$\exists x_0 \in \G$, $x_1, x_2 \in \P$ with $x_0 = x_1 x_2$}\\
                                                    &\text{ and $x_j \sim \exp(2\pi i X_j)$}\}.
\end{align*}

\begin{definition} Let $X,Y$ be any sets and $Q \subseteq X \times Y$. Call $y \in Y$ a \emph{fat point} for $Q$ with respect to the projection $\pi_X : X \times Y \to X$ if $\pi_X(Q) \times \{y\} \subseteq Q$; that is, if for all $(x,y') \in Q$ we have $(x,y) \in Q$. Let $Q//\pi_X \subseteq Y$ denote the set of fat points for $Q$ with respect to $\pi_X$. \end{definition}

\begin{example}
Let $X = Y = \RR$ and let $Q$ be the convex hull of $(0,0), (0,1), (1,1), (\tfrac{3}{2}, \tfrac{1}{2}), (1,0)$:
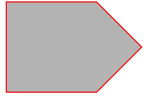
\begin{center}
\begin{tikzpicture}[scale=1.2]
\filldraw[color=black!30!white] (0,0) -- (0,1) -- (1,1) -- (1.5,0.5) -- (1,0) -- (0,0);
\draw[color=red] (0,0) -- (0,1) -- (1,1) -- (1.5,0.5) -- (1,0) -- (0,0);
\end{tikzpicture}
\end{center}
Then $Q//\pi_X = \{\tfrac{1}{2}\}$ and $Q//\pi_Y = [0,1]$.
\end{example}

If $X \in \overline{\alc(\G)}$, let $\widetilde{X}$ denote the unique element of $\overline{\alc(\G)}$ such that $\widetilde{X}$ is $\Ad(\G)$-conjugate to $-X$, i.e.\ such that $\exp(2\pi i \widetilde{X}) \sim \exp(-2\pi i X)$.
\begin{lemma} \label{lem:B-fat-points} Assume $\G$ is simply connected. Let $\pi_1$ be the projection onto the first coordinate of $\overline{\alc(\G,\K)}^3$. Then $\B(\G,\K) = \{X \in \alc : (\widetilde{X},X) \in \bigP(\G,\K)// \pi_1\}$.
\end{lemma}

\begin{proof}
    Let $R$ be the set that we are trying to prove is equal to $\B(\G,\K)$. By definition, $R$ is the set of of $X \in \overline{\alc}$ such that $(Y,\widetilde{X},X) \in \bigP(\G,\K)$ for all $Y \in \overline{\alc}$.

    Set $x = \exp(\pi i X)$ where $X \in \B(\G,\K)$. Take $Y \in \overline{\alc(\G,\K)}$ and set $y = \exp(\pi i Y)$. Then $x\K x^{-1} \cap \K y \K \neq \emptyset$, so by Lemma~\ref{lem:KAK-conj} there exists $k \in \K$ with $xkx^{-1}\theta(xkx^{-1})^{-1} \sim y\theta(y)^{-1} = \exp(2\pi i Y)$, i.e.
    \begin{align} \label{eq:conj}
        &xkx^{-1}\theta(xkx^{-1})^{-1} = xkx^{-2}k^{-1}x \sim kx^{-2}k^{-1}x^2 \nonumber\\
        \Rightarrow\, &\exp(2\pi i Y) \sim (k\exp(2\pi i \widetilde{X})k^{-1}) \cdot \exp(2\pi i X).
    \end{align}
    This says $(Y,\widetilde{X},X) \in \bigP(\G,\K)$, so $X \in R$ since $Y$ was arbitrary.

    Conversely, suppose $X \in R$, meaning that for any $Y \in \alc(\G,\K)$ we have $\exp(2\pi i Y) \sim p_1 p_2$ for $p_1, p_2 \in \P$ with $p_1 \sim x^{-2}, p_2 \sim x^{2}$ where $x = \exp(\pi i X)$. By Theorem~\ref{thm:cartan-facts}, we can write $p_j = k_j \exp(\pi i A_j)k_j^{-1}$ where $A_j \in \overline{\alc(\G,\K)}$ and $k_j \in \K$. Then
    \begin{equation*}
        p_1 = k_1 \exp(2\pi i A_1) k_1^{-1} \sim x^{-2} = \exp(2\pi i \widetilde{X}),
    \end{equation*}
    forcing $A_1 = \widetilde{X}$ by Lemma~\ref{lem:KAK-conj} and Theorem~\ref{thm:fund-alcove}. Similarly, $A_2 = X$. Now 
    \begin{equation*}
        \exp(2\pi i Y) \sim p_1 p_2 = (k_1 x^{-2} k_1^{-1})(k_2 x^2 k_2^{-1}) \sim (k_2^{-1} k_1 x^{-2} k_1^{-1} k_2)\cdot x^2.
    \end{equation*}
    Note that this is the same expression as \eqref{eq:conj}. We can now reverse the arguments in the previous paragraph, starting from \eqref{eq:conj}, to deduce that $x\K x^{-1} \cap \K y \K \neq \emptyset$ for all $y \in \A$ and hence $X \in \B(\G,\K)$.
\end{proof}

Generalizing work of Agnihotri and Woodward in the $\G = \SU(n)$ case \cite{agnihotri-woodward}, Teleman and Woodward proved that, remarkably, the set $\bigP(\G)$ is a convex polytope described by explicit (if complicated) inequalities \cite{teleman-woodward}. Since $\bigP(\G,\K) \subseteq \bigP(\G) \cap \overline{\alc(\G,\K)}^3$, this implies some linear inequalities which the points of $\bigP(\G,\K)$ must satisfy. In turn, the next lemma shows how these inequalities imply linear inequalities on $\bigP(\G,\K) // \pi_1$.
\begin{lemma}
    Let $Q,R_1,R_2$ be convex polytopes with $Q \subseteq R_1 \times R_2$, and let $\pi_1, \pi_2$ be the projections onto the two factors. Then
    \begin{equation} \label{eq:fat-point-lem}
        Q // \pi_1 = \bigcap_v \pi_2((\{v\} \times R_2) \cap Q)
    \end{equation}
    where $v$ runs over the vertices of $\pi_1(Q)$. In particular, $Q // \pi_1$ is again a polytope.
\end{lemma}

\begin{proof}
    By definition, if $r_2 \in Q//\pi_1$ and $r_1 \in \pi_1(Q)$ then $r_2 \in \pi_2((\{r_1\} \times R_2) \cap Q)$. Conversely, if $r_2$ is in the right-hand side of \eqref{eq:fat-point-lem}, then for every vertex $v$ of $\pi_1(Q)$ we have $(v,r_2) \in Q$. But then $Q$ contains the convex hull of these points, namely $\pi_1(Q) \times \{r_2\}$.
\end{proof}

\begin{theorem} \label{thm:B-polytope}
    $\B(\G,\K)$ is contained in the polytope
    \begin{equation*}
        \{X \in \overline{\alc} : (\widetilde{X},X) \in (\bigP(\G) \cap \overline{\alc(\G,\K)}) // \pi\}
    \end{equation*}
    where $\pi$ is projection onto the first factor of $\overline{\alc(\G,\K)}^3$.
\end{theorem}

In every case in which we are able to describe $\B(\G,\K)$, it turns out that the containment of Theorem~\ref{thm:B-polytope} is actually an equality. Given this, it seems reasonable to suspect that the containment $\bigP(\G,\K) \subseteq \bigP(\G) \cap \overline{\alc(\G,\K)}$ is also actually an equality. In the case $\G = \SU(n), \K = \SO(n)$, this has been proven by Falbel and Wentworth \cite{falbel-wentworth}, which we will use to compute $\B(\SU(n),\SO(n))$ exactly in the next section.

\section{Type AI: $\G = \SU(n)$, $\K = \SO(n)$} \label{sec:SU-SO}
We have worked out some details of this case in previous examples, but to summarize:
\begin{itemize}
    \item $\mathfrak{k} = \mathfrak{so}(n)$, the space of real skew-symmetric matrices.
    \item $\mathfrak{p}$ consists of the matrices $iS \in \mathfrak{su}(n)$ with $S$ real symmetric.
    \item Take $\mathfrak{a} \subseteq \mathfrak{p}$ as the matrices $iD$ with $D$ real diagonal of trace 0, so $\mathfrak{h} = \mathfrak{a}$.
    \item The restricted roots are the same as the usual roots $\epsilon_p-\epsilon_q$ of $\mathfrak{g} \otimes \CC$.
    \item Take $\{\epsilon_p-\epsilon_q : 1 \leq p < q \leq n\}$ as positive roots, and simple roots $\alpha_i = \epsilon_i - \epsilon_{i+1}$.
    \item The Stiefel diagram is $\bigcup_{n \in \ZZ, p \neq q} \{\mathbf{x} \in \RR^n : x_p-x_q = n, \sum x_j = 0\}$.
    \item Take $\alc = \{\mathbf{x} \in \RR^n : x_1 > \cdots > x_n > x_1 - 1, \sum x_j = 0\}$ as a fundamental alcove.
\end{itemize}

\subsection{The group $C(\su(n), \so(n))$}
Let us work out in detail what Theorem~\ref{thm:extW-stab} says concretely. The highest root is $-\alpha_0 = \epsilon_1 - \epsilon_n = \alpha_1 + \cdots + \alpha_n$, and $a_i = 1$ for all $i$. The fundamental coweights are
\begin{equation*}
    \omega_k = (\overbrace{\tfrac{k}{n}, \ldots, \tfrac{k}{n}}^{n-k}, \overbrace{\tfrac{k-n}{n}, \ldots, \tfrac{k-n}{n}}^{k}), \quad k = 1, \ldots, n-1;
\end{equation*}
recall we also set $\omega_0 = 0$. A permutation $w \in W \simeq S_n$ has a cyclic descent at $i$ if $w(\alpha_i) = \epsilon_{w(i)} - \epsilon_{w(i+1)}$ is a negative root, i.e. $w(i) > w(i+1)$. Note that this is still correct in the case $i = 0$ if we interpret $w(0)$ to mean $w(n)$. Therefore $C$ consists of the permutations with exactly one cyclic descent, i.e.\ those in the cyclic group generated by the long cycle $c = 23\cdots n1 = (1\,2\, \cdots \, n)$.

\begin{lemma} \label{lem:AI-fixed-point}
The only point of $\overline{\alc(\SU(n),\SO(n))}$ fixed by $\extW_{\alc}$ is the centroid 
\begin{equation*}
    \frac{\omega_0 + \cdots + \omega_{n-1}}{n} = (\tfrac{n-1}{2n}, \tfrac{n-3}{2n}, \ldots, -\tfrac{n-3}{2n}, -\tfrac{n-1}{2n}).
\end{equation*}
\end{lemma}
\begin{proof}
    By Theorem~\ref{thm:extW-stab}, $\extW_\alc$ is generated by $f = c\tau_{-\delta_c} = c\tau_{-\omega_{n-1}}$. To calculate the action of $f$ on the fundamental alcove $\alc$, it suffices to compute its action on the vertices $\omega_i$:
\begin{align*}
    \alpha_j(f(\omega_i)) &= \alpha_j(c(\omega_i - \omega_{n-1})) = (c^{-1}\alpha_j)(\omega_i - \omega_{n-1}) = \alpha_{j-1}(\omega_i - \omega_{n-1})\\
    &= \delta_{i+1,j} \quad \text{(for $1 \leq j < n$)}.
\end{align*}
Since the $\alpha_j$ are a dual basis to the $\omega_j$ by definition, this shows that $f(\omega_i) = \omega_{i+1}$. As $\omega_1, \ldots, \omega_{n-1}$ are linearly independent, the only fixed point of $\extW_{\alc}$ acting on $\alc$ is the centroid $\tfrac{1}{n}(\omega_0 + \cdots + \omega_{n-1})$.
\end{proof}

Combining this lemma with Theorem~\ref{thm:main-1} gets us one direction of Theorem~\ref{thm:main-2} in the type AI case:
\begin{corollary} \label{cor:AI-necc}
    If $U,V \in \SU(n)$ have $\SO(n) \cdot U \cdot \SO(n) \cdot V \cdot \SO(n) = \SU(n)$, then $UU^T$ and $VV^T$ both have spectrum $e^{\pi i (n-2j+1)/n}$ for $j = 1, \ldots, n$. Equivalently, $UU^T$ and $VV^T$ have characteristic polynomial $x^n + (-1)^{n}$.
\end{corollary}

\subsection{The polytope $\bigP(\SU(n), \SO(n))$} \label{subsec:P-AI}
To prove the converse of Corollary~\ref{cor:AI-necc}, we apply Lemma~\ref{lem:B-fat-points}, for which we need some knowledge of $\bigP(\SU(n), \SO(n))$. We start with an explicit description of the polytope $\bigP(\SU(n))$. Given a vector $\mathbf{x} \in \RR^n$ and $I \subseteq [n]$, write $\mathbf{x}_I$ for $\sum_{i \in I} x_i$.

\begin{theorem}[\cite{agnihotri-woodward}] \label{thm:qLR} The polytope $\bigP(\SU(n))$ is the set of $(X,Y,Z) \in \overline{\alc}(\SU(n))^3$ obeying every inequality $-X_K + Y_I + Z_J \leq d$ for which the quantum Littlewood-Richardson coefficient $c_{IJ}^{K,d}$ is nonzero, where $I,J,K \subseteq [n]$ are subsets of equal size and $d \geq 0$ is an integer. \end{theorem}

\begin{lemma} \label{lem:AI-suff}
Let $\zeta$ be the centroid of $\alc(\su(n),\so(n))$, so $\zeta_j = \tfrac{n-2j+1}{2n}$ for $j = 1, \ldots, n$. Then $(X,\zeta,\zeta) \in \bigP(\SU(n))$ for any $X \in \overline{\alc(\su(n),\so(n))}$.
\end{lemma}

\begin{proof}
By Theorem~\ref{thm:qLR}, we must check the inequality 
\begin{equation} \label{eq:zeta-ineq-1}
    -X_K + \zeta_I + \zeta_J \leq d
\end{equation}
whenever $c_{IJ}^{K,d} > 0$ and $X \in \overline{\alc}$. It suffices to check this when $X$ is a vertex
\begin{equation*}
    \omega_{n-p} = (\overbrace{\tfrac{n-p}{n}, \ldots, \tfrac{n-p}{n}}^{p}, \overbrace{\tfrac{-p}{n}, \ldots, \tfrac{-p}{n}}^{n-p}).
\end{equation*}
of $\overline{\alc}$.
In this case, \eqref{eq:zeta-ineq-1} reads 
\begin{equation} \label{eq:zeta-ineq-2}
   -\frac{1}{n}((n-p)|K \cap [p]| - p|K \cap [p+1,n]|) + \sum_{i \in I} \frac{n-2i+1}{2n} + \sum_{j \in J} \frac{n-2j+1}{2n} \leq d
\end{equation}

Setting $k = |I| = |J| = |K|$ and $a = |K \cap [p]|$ and rearranging, \eqref{eq:zeta-ineq-2} becomes
\begin{equation} \label{eq:zeta-ineq-3}
    pk - na + k(n+1) \leq nd + \sum I + \sum J.
\end{equation}
If $c_{IJ}^{K,d} > 0$, then
\begin{equation} \label{eq:grading}
nd + \sum I + \sum J = \sum K + \sum_{i=n-k+1}^{n} i
\end{equation} by Proposition~\ref{prop:grading}. We now prove that \eqref{eq:grading} implies \eqref{eq:zeta-ineq-3}.

We must show that 
\begin{equation*}
nd + \sum I + \sum J - pk + na - k(n+1) \geq 0.
\end{equation*}
Using \eqref{eq:grading}, the left side here is 
\begin{align*}
na - pk + \sum K - \sum_{i=1}^k i.
\end{align*}
Write $K = \{K_1 < \cdots < K_k\}$. Then $K_i \geq i$ for $i = 1, \ldots, a = |K \cap [p]|$ and $K_i \geq p+i-a$ for $i = a+1, \ldots, k$. These inequalities give 
\begin{equation*}
na - pk + \sum_{i=1}^k (K_i-i) \geq na - pk + \sum_{i=a+1}^k (p-a) = a(n-k-p+a).
\end{equation*}
But $n-k-p+a \geq 0$ because it is the cardinality of the set $([n] \setminus K) \setminus [p]$.
\end{proof}

\begin{theorem*}[Theorem~\ref{thm:main-2}, type AI case]
$\B(\SU(n),\SO(n))$ is the singleton $\{\zeta\}$ where $\zeta_j = \tfrac{n-2j+1}{2n}$. Equivalently, $\SO(n) \cdot U \cdot \SO(n) \cdot V \cdot \SO(n) = \SU(n)$ if and only if $UU^T$ and $VV^T$ both have spectrum $e^{\pi i (n-2j+1)/n}$ for $j = 1, \ldots, n$, i.e.\ characteristic polynomial $x^n + (-1)^{n}$. 
\end{theorem*}

\begin{proof}
The statement in terms of eigenvalues is equivalent to $\B(\SU(n),\SO(n)) = \{\zeta\}$ by Lemma~\ref{lem:KAK-conj}. Corollary~\ref{cor:AI-necc} shows that $\B(\SU(n),\SO(n)) \subseteq \{\zeta\}$. Lemma~\ref{lem:AI-suff} says $(\zeta,\zeta) \in \bigP(\SU(n))//\pi_1$, and a nontrivial result of Falbel and Wentworth asserts that $\bigP(\SU(n)) = \bigP(\SU(n), \SO(n))$ \cite{falbel-wentworth}. Since $\exp(i\pi \zeta)$ is self-inverse, $\widetilde{\zeta} = \zeta$, so $(\widetilde{\zeta},\zeta) \in \bigP(\SU(n), \SO(n))//\pi_1$. By Lemma~\ref{lem:B-fat-points}, this is equivalent to $\zeta \in \B(\SU(n),\SO(n))$.
\end{proof}

\section{Type AII: $\G = \SU(2n), \K = \Sp(n)$}
\label{sec:SU-Sp}

The compact symplectic group $\Sp(n)$ is the fixed-point subgroup of the involution
\begin{equation*}
    \theta\left(\begin{bmatrix} A & B \\ C & D \end{bmatrix} \right) = \begin{bmatrix} 0 & I \\ -I & 0 \end{bmatrix} \begin{bmatrix} A & B \\ C & D \end{bmatrix}\begin{bmatrix} 0 & I \\ -I & 0 \end{bmatrix}^{-1} =  \begin{bmatrix} \overline{D} & -\overline{C} \\ -\overline{B} & \overline{A} \end{bmatrix},
\end{equation*}
on $\SU(2n)$. Explicitly, it is the set of unitary matrices of the form $\left[\begin{smallmatrix} X & -\overline{Y} \\ Y & \overline{X} \end{smallmatrix}\right]$. Now
\begin{itemize}
    \item $\mathfrak{k} = \mathfrak{sp}(n)$, the space of matrices $\left[ \begin{smallmatrix} A & -\overline{B} \\ B & \overline{A} \end{smallmatrix} \right]$ with $A$ skew-Hermitian and $B$ (complex) symmetric.
    \item $\mathfrak{p}$ is the space of matrices $\left[ \begin{smallmatrix} A & \overline{C} \\ C & -\overline{A} \end{smallmatrix} \right]$ with $A$ skew-Hermitian of trace 0 and $B$ (complex) skew-symmetric.
    \item We can take $\mathfrak{a}$ to be the diagonal elements of $\mathfrak{p}$, i.e.\ diagonal matrices with diagonal of the form $i\lambda_1, \ldots, i\lambda_{n}, i\lambda_1, \ldots, i\lambda_{n}$ with $\sum_j \lambda_j = 0$ and all $\lambda_j$ real. We can then once again take $\mathfrak{h}$ to be the diagonal matrices in $\mathfrak{su}(2n)$.
    \item The restricted roots are $\epsilon_p - \epsilon_q$ with $p-q \notin \{0, \pm n\}$.
\end{itemize}

Identify $i\diag(x_1, \ldots, x_n, x_1, \ldots, x_n)$ with $(x_1, \ldots, x_n) \in \RR^n$, so $\mathfrak{a} = \{\mathbf{x} \in \RR^{n} : \sum_j x_j = 0\}$. Then the restricted roots are just the usual type $A_{n-1}$ roots $\epsilon_p-\epsilon_q$ for $p \neq q$, and we can reduce to the arguments in \sectionsymbol\ref{sec:SU-SO} without much trouble.

\begin{theorem}[\ref{thm:main-2}, type AII case]
    $\B(\SU(2n), \Sp(n)) = \{\zeta\}$ where $\zeta_j = \tfrac{n-2j+1}{2n}$ for $j = 1, \ldots, n$. That is, $\Sp(n) \cdot U \cdot \Sp(n) \cdot V \cdot \Sp(n) = \SU(2n)$ if and only if $U\theta(U)^{-1}$ and $V\theta(V)^{-1}$ both have spectrum $e^{\pi i (n-2j+1)/n}$ for $j = 1, \ldots, n$ with each eigenvalue having multiplicity $2$. Equivalently, $U\theta(U)^{-1}$ and $V\theta(V)^{-1}$ have characteristic polynomial $(x^n + (-1)^{n})^2$.
\end{theorem}

\begin{proof}
Since the restricted root system is the same as in the type AI case, Lemma~\ref{lem:AI-fixed-point} shows equally well that $\B(\SU(2n), \Sp(n)) \subseteq \{\zeta\}$. For the converse, let $D \in \SU(n)$ be diagonal with diagonal entries $\exp(\pi i \zeta_1), \ldots, \exp(\pi i \zeta_n)$. Let $\Delta(D)$ denote the block diagonal matrix with blocks $D,D$, so $\Delta(D) \in \A$. Let $\H = \Delta(\SO(n))$, a subgroup of $\K = \Sp(n)$. Now
\begin{align*}
\K \cdot \Delta(D) \cdot \K \cdot \Delta(D)^{-1} \cdot \K &=  \K \cdot \H \Delta(D)\H  \cdot \K  \cdot \H\Delta(D)^{-1}\H \cdot \K\\ 
&= \K \cdot \Delta(\SO(n)D\SO(n)D\SO(n)) \cdot \K\\
&= \K \cdot \Delta(\SU(n)) \cdot \K \qquad \text{(by the type AI case of Theorem~\ref{thm:main-2})}\\
&\supseteq \K\A\K = \SU(2n).
\end{align*}
This shows $\zeta \in \B(\SU(2n), \Sp(n))$.
\end{proof}

\section{$\G = \SU(2n)$, $\K = \operatorname{S}(\U(n) \times \U(n))$} \label{sec:SU-block}
Here $\K$ is the set of elements of $\SU(n)$ of the form $\left[\begin{smallmatrix} A & 0 \\ 0 & D \end{smallmatrix}\right]$ where $A, D$ are $n \times n$, the fixed points of the involution
\begin{equation*}
    \theta\left(\begin{bmatrix} A & B \\ C & D \end{bmatrix} \right) = \begin{bmatrix} I & 0 \\ 0 & -I \end{bmatrix}\begin{bmatrix} A & B \\ C & D \end{bmatrix}\begin{bmatrix} I & 0 \\ 0 & -I \end{bmatrix}^{-1} = \begin{bmatrix} A & -B \\ -C & D \end{bmatrix}.
\end{equation*}
on $\SU(2n)$.

Now 
\begin{itemize}
    \item $\mathfrak{k} = \mathfrak{s}(\mathfrak{u}(n) \oplus \mathfrak{u}(n))$, the space of matrices $\left[ \begin{smallmatrix} A & 0 \\ 0 & D \end{smallmatrix} \right]$ with $A,D$ skew-Hermitian and $\tr(A) + \tr(D) = 0$.
    \item $\mathfrak{p}$ is the space of matrices $\left[ \begin{smallmatrix} 0 & C \\ -C^{\dagger} & 0 \end{smallmatrix} \right]$ with $C$ any $n \times n$ complex matrix.
    \item We can take $\mathfrak{a}$ to be the matrices $\left[ \begin{smallmatrix} 0 & iD \\ iD & 0 \end{smallmatrix} \right]$ with $D$ real diagonal. Thus $\h$ can \emph{not} be the space of diagonal matrices as before. Instead, we take
    \begin{equation*}
    \h = \left\{i\left[\begin{smallmatrix} E & F \\ F & E \end{smallmatrix}\right] : \text{$E,F$ real diagonal, $\tr(E)=0$}\right\}.
    \end{equation*} 
    This is a maximal abelian subalgebra of $\mathfrak{su}(n)$ containing $\mathfrak{a}$.
    \item The roots and root spaces of $\mathfrak{g} \otimes \CC$ with respect to $\mathfrak{h} \otimes \CC$ are  
    \begin{center}
    \begin{tabular}{c|c}
        root & root space \\
        \hline 
        $\epsilon_i - \epsilon_j + \phi_i - \phi_j$ & $\CC(e_{ij}^{\scriptscriptstyle \nwarrow} + e_{ij}^{\scriptscriptstyle \searrow} + e_{ij}^{\scriptscriptstyle \nearrow} + e_{ij}^{\scriptscriptstyle \swarrow})$\\
        \hline 
        $\epsilon_i - \epsilon_j - \phi_i + \phi_j$ & $\CC(e_{ij}^{\scriptscriptstyle \nwarrow} + e_{ij}^{\scriptscriptstyle \searrow} - e_{ij}^{\scriptscriptstyle \nearrow} - e_{ij}^{\scriptscriptstyle \swarrow})$\\
        \hline 
        $\epsilon_i - \epsilon_j + \phi_i + \phi_j$ & $\CC(e_{ij}^{\scriptscriptstyle \nwarrow} -  e_{ij}^{\scriptscriptstyle \searrow} - e_{ij}^{\scriptscriptstyle \nearrow} + e_{ij}^{\scriptscriptstyle \swarrow})$ \\
        \hline 
        $\epsilon_i - \epsilon_j - \phi_i - \phi_j$ & $\CC(e_{ij}^{\scriptscriptstyle \nwarrow} - e_{ij}^{\scriptscriptstyle \searrow} + e_{ij}^{\scriptscriptstyle \nearrow} - e_{ij}^{\scriptscriptstyle \swarrow})$
    \end{tabular}
\end{center}
where 
\begin{itemize}
\item $\epsilon_i$ and $\phi_i$ are the linear functionals on $\mathfrak{h} \otimes \CC$ sending $\left[\begin{smallmatrix} E & F \\ F & E \end{smallmatrix}\right]$ to $E_{ii}$ and $F_{ii}$ respectively;
\item $e_{ij}$ is the $n \times n$ matrix with a $1$ in entry $(i,j)$ and $0$'s elsewhere;
\item if $M$ is $n \times n$, then $M^{\scriptscriptstyle \nearrow}$ is the $2n \times 2n$ matrix $\left[ \begin{smallmatrix} 0 & M \\ 0 & 0 \end{smallmatrix} \right]$, defining $M^{\scriptscriptstyle \nwarrow}$, $M^{\scriptscriptstyle \searrow}$, and $M^{\scriptscriptstyle \swarrow}$ analogously.
\end{itemize}
\item Identify $\left[ \begin{smallmatrix} 0 & iD \\ iD & 0 \end{smallmatrix} \right] \in \a$ with $(D_{11}, \ldots, D_{nn}) \in \RR^n$. Note that this identifies $\a$ with all of $\RR^n$, \emph{not} just the sum 0 hyperplane as in the type AI and AII cases.
\item The restricted roots $\Phi(\su(2n), \mathfrak{s}(\u(n) \oplus \u(n)))$ are $\phi_i - \phi_j$ ($i \neq j$) and $\pm (\phi_i + \phi_j)$ (any $i,j$).  Thus the restricted root system is of type $C_n$. We take $\alpha_i = \phi_i - \phi_{i+1}$ for $i = 1, \ldots, n-1$ and $\alpha_n = 2\phi_n$ to be the simple roots, so the positive roots are $\phi_i - \phi_j$ for $i < j$ and $\phi_i + \phi_j$ for all $i,j$. The highest root is $-\alpha_0 = 2\phi_1 = 2\alpha_1 + \cdots + 2\alpha_{n-1} + \alpha_n$.
\item The Stiefel diagram is $\bigcup_{1 \leq i,j \leq n} \{\mathbf{x} \in \RR^n : x_i \pm x_j \in \ZZ\}$, and we take a fundamental alcove to be $\alc = \{\mathbf{x} \in \RR^n : \tfrac{1}{2} > x_1 > \cdots > x_n > 0\}$. The fundamental coweights are $\omega_i = e_1 + \cdots + e_i$ for $i < n$ and $\omega_n = \tfrac{1}{2}(e_1 + \cdots + e_n)$, and the integers $a_0, a_1, \ldots, a_{n-1}, a_n$ are $1, 2, \ldots, 2, 1$.
\end{itemize}

The finite Weyl group $W$ is generated by the reflections $s_{\alpha_i,0}$ ($i < n$), which as in the type AI case swap coordinates $i$ and $i+1$, as well as the reflection $s_{\alpha_n,0}$ across $\{x_n = 0\}$, which negates coordinate $n$. Thus $W$ is the \emph{hyperoctahedral group} $B_n$, the group of \emph{signed permutations} $w_1 \cdots w_n$ where $|w_1| \cdots |w_n|$ is a permutation of $n$. For instance, $\bar{4}\bar{3}12 \in B_4$ where $\bar{4} = -4$. The action of $W$ on $\mathfrak{a}$ is given by $w(e_i) = \sgn(w(i)) e_{|w(i)|}$.

With this description, $w \in B_n$ has a cyclic descent at $0 < i < n$ if $s w(i) > s w(i+1)$ where $s = \sgn(w(i))\sgn(w(i+1))$, a cyclic descent at $n$ if $w(n) < 0$, and a cyclic descent at $0$ if $w(1) > 0$.
\begin{proposition} $C(\su(2n), \mathfrak{s}(\u(n) \oplus \u(n)))$ is the group of order 2 generated by $c = \bar{n} \cdots \bar{2}\bar{1}$, which acts on $\alc$ by the map $(x_1, \ldots, x_n) \mapsto (\tfrac{1}{2}-x_n, \ldots, \tfrac{1}{2}-x_1)$.   \end{proposition}
\begin{proof}
    Since $\cdes(w) = \sum_{i \in \cDes(w)} a_i$ and $a_0, a_1, \ldots, a_{n-1}, a_n = 1, 2, \ldots, 2, 1$, we can only have $\cdes(w) = 1$ if $\cDes(w)$ is $\{0\}$ or $\{n\}$. Suppose $\cDes(w) = \{0\}$. Then $w(1), w(n) > 0$ and $w(2), \ldots, w(n-1)$ must all be positive because otherwise there would be a descent $w(i+1) < 0 < w(i)$. But then we must have $0 < w(1) < \cdots < w(n)$ since there are no descents in $\{1, \ldots, n\}$, i.e.\ $w = 12\cdots n$. If $w \in C$ has its unique cyclic descent at $n$, the same argument holds with all signs reversed, forcing $w = c = \bar{n} \cdots \bar{2} \bar{1}$.

To see that $c\tau_{-\delta_c} = c\tau_{-\omega_n}$ acts on $\overline{\alc}$ as claimed, apply it to the vertices $a_i^{-1} \omega_i$. First, since $c(e_i) = -e_{n-i+1}$ we have $c(\alpha_j) = -e_{n-j+1} + e_{n-j} = \alpha_{n-j}$ for $j = 0, \ldots, n$. Apply a simple root $\alpha_j$ ($j > 0$), recalling that they are dual to the $\omega_j$:
\begin{align*}
    \alpha_j(c\tau_{-\delta_c} (a_i^{-1}\omega_i)) &= (c^{-1}\alpha_j)(a_i^{-1} \omega_i - \omega_n) = \alpha_{n-j}(a_i^{-1}\omega_i - \omega_n)\\
    &= a_i^{-1}\delta_{j,n-i} = a_{n-i}^{-1} \delta_{j,n-i}.
\end{align*}
Note that this formula holds even for $j = n$. It follows that $c\tau_{-\delta_c}$ maps $a_i^{-1}\omega_i$ to $a_{n-i}^{-1}\omega_{n-i}$. The explicit formula $a_i^{-1}\omega_i = \tfrac{1}{2}(e_1 + \cdots + e_i)$ shows that the map $(x_1, \ldots, x_n) \mapsto (\tfrac{1}{2}-x_n, \ldots, \tfrac{1}{2}-x_1)$ acts on the vertices in the same way. Since both maps are affine linear, preserve $\overline{\alc}$, and have the same action on its vertices, they must be the same.
\end{proof}

\begin{corollary} \label{cor:SU-block-necc}
    If $U,V \in \SU(2n)$ have $\K U \K V \K = \SU(2n)$ where $\K = \operatorname{S}(\U(n) \times \U(n))$, then $U\theta(U)^{-1}$ and $V\theta(V)^{-1}$ both have the same spectrum $e^{\pm \pi i x_1}, \ldots, e^{\pm \pi i x_n}$ where $\tfrac{1}{2} \geq x_1 \geq \cdots \geq x_n \geq 0$ and $(x_1, \ldots, x_n) =  (\tfrac{1}{2}-x_n, \ldots, \tfrac{1}{2}-x_1)$.
\end{corollary}

There is a different interpretation of the canonical parameters $a(U)$ which will be useful.
\begin{proposition} \label{prop:SVD}
    For $U \in \SU(2n)$, we have
    \begin{equation*}
    a(U) = \tfrac{1}{\pi}(\cos^{-1} \sigma_n(U_{11}), \ldots, \cos^{-1} \sigma_1(U_{11}))
    \end{equation*}
    where $\sigma_i(M)$ denotes the $i$\textsuperscript{th} largest singular value of $M$ and $U_{11}$ is the upper-left $n \times n$ corner of $U$.
\end{proposition}

\begin{proof}
    When $D$ is real diagonal we have $\exp\left(\left[ \begin{smallmatrix} 0 & iD \\ iD & 0 \end{smallmatrix} \right]\right) = \left[ \begin{smallmatrix} \cos D & i \sin D \\ i \sin D & \cos D \end{smallmatrix}\right]$. The resulting Cartan decomposition of Theorem~\ref{thm:cartan-facts}(b) is the \emph{cosine-sine decomposition} 
    \begin{equation*}
        U = \begin{bmatrix} P & 0 \\ 0 & Q \end{bmatrix} \begin{bmatrix} \cos D & i \sin D \\ i \sin D & \cos D  \end{bmatrix} \begin{bmatrix} R & 0 \\ 0 & S \end{bmatrix}
    \end{equation*}
    where $D_{ii} = \pi a_i(U)$ and $P,Q,R,S$ are unitary. This gives $U_{11} = P\cos(D)R$, so the singular values of $U_{11}$ are the numbers $\cos(\pi a_i(U)) = \cos(D_{ii})$. More specifically, since $a(U) \in \overline{\alc}$ we have $0 \leq \cos(\pi a_1) \leq \cdots \leq \cos(\pi a_n)$, and so $\cos(\pi a_i(U)) = \sigma_{n-i+1}(U_{11})$.
\end{proof}

This gives the following restatement of Corollary~\ref{cor:SU-block-necc}.
\begin{corollary}
    If $U,V \in \SU(2n)$ have $\K U \K V\K = \SU(2n)$ where $\K = \operatorname{S}(\U(n) \times \U(n))$, then the upper left $n \times n$ corners of $U, V$ have the same singular values $\sigma_1 \geq \cdots \geq \sigma_n$ and they satisfy the equations $\sigma_i^2 + \sigma_{n-i+1}^2 = 1$.
\end{corollary}

As in \sectionsymbol\ref{sec:SU-SO} and \sectionsymbol\ref{sec:SU-Sp}, these equations also turn out to be sufficient to guarantee $\K U \K V \K = \SU(2n)$, but now an inductive approach is available: the truth of the general statement will follow from the $n = 1$ and $n = 2$ cases, which are explicit calculations. For both calculations, it will be useful to recall that if $\mathbf{x} \in \overline{\alc}$ and $D$ is diagonal with diagonal $\pi \mathbf{x}$, then $g = \exp\left(\left[ \begin{smallmatrix} 0 &  i D \\  i D & 0 \end{smallmatrix} \right]\right) = \left[ \begin{smallmatrix} \cos D & i \sin D \\ i \sin D & \cos D \end{smallmatrix}\right]$ is the unique element of $\exp(\pi i \overline{\alc})$ with $a(g) = \mathbf{x}$.

\begin{lemma} \label{lem:SU2-suff}
    $\B(\SU(2),\operatorname{S}(\U(1) \times \U(1)))$ contains the point $1/4$.
\end{lemma}

\begin{proof}
    Set $U = \frac{1}{\sqrt{2}}\left[ \begin{smallmatrix} 1 & i \\ i & 1 \end{smallmatrix} \right]$, so $a(U) = 1/4$.
    We must show that any element of $\SU(2)$ can be written 
    \begin{align*}
        &\begin{bmatrix} e^{i\alpha} & 0 \\ 0 & e^{-i\alpha} \end{bmatrix}
        U
        \begin{bmatrix} e^{i\beta} & 0 \\ 0 & e^{-i\beta} \end{bmatrix}
        U
        \begin{bmatrix} e^{i\gamma} & 0 \\ 0 & e^{-i\gamma} \end{bmatrix} \\
        =\, &\begin{bmatrix} e^{i\alpha} & 0 \\ 0 & e^{-i\alpha} \end{bmatrix} i \begin{bmatrix} \sin \beta & \cos \beta \\ \cos \beta & -\sin \beta \end{bmatrix}
        \begin{bmatrix} e^{i\gamma} & 0 \\ 0 & e^{-i\gamma} \end{bmatrix} \\
        &= i \begin{bmatrix} e^{i(\alpha+\gamma)}\sin \beta & e^{i(\alpha-\gamma)}\cos \beta \\ e^{-i(\alpha-\gamma)}\cos \beta & -e^{-i(\alpha+\gamma)}\sin \beta \end{bmatrix}.
    \end{align*}
    This is easily seen to be equivalent to the more standard form $\left[ \begin{smallmatrix} e^{i\phi} \sin \beta & -e^{-i\psi} \cos \beta \\ e^{i\psi} \cos \beta & e^{-i\phi} \sin \beta \end{smallmatrix} \right]$.
\end{proof}


\begin{lemma} \label{lem:SU4-suff}
    $\B(\SU(4),\operatorname{S}(\U(2) \times \U(2)))$ contains the line $\{(\tfrac{1}{2}-x,x) : x \in [0,\tfrac{1}{4}]\}$.
\end{lemma}

\begin{proof}
    Fix $x \in [0, \tfrac{1}{4}]$, and let $D$ be diagonal with diagonal entries $\pi(\tfrac{1}{2}-x), \pi x$ and $V = \left[ \begin{smallmatrix} \cos D & i \sin D \\ i \sin D & \cos D \end{smallmatrix}\right]$. We must show that $\K V \K V^{-1} \K = \SU(4)$, or equivalently that $a(VkV^{-1})$ can take any value in the fundamental alcove $\overline{\alc}$ with an appropriate choice of $k \in \K$. Writing $k = \left[ \begin{smallmatrix} K_1 & 0 \\ 0 & K_2 \end{smallmatrix} \right]$, this is equivalent by Proposition~\ref{prop:SVD} to showing that the upper-left corner of $VkV^{-1}$, namely
    \begin{equation*}
        M = \cos(D)K_1 \cos(D) - \sin(D) K_2 \sin(D),
    \end{equation*}
    can have any possible pair of singular values $1 \geq \sigma_1 \geq \sigma_2 \geq 0$ with an appropriate choice of $K_1,K_2 \in \U(n)$ with $\det(K_1 K_2) = 1$. In fact, since the whole equation can be multiplied by a phase without changing $\sigma_i$, the assumption $\det(K_1 K_2) = 1$ can be dispensed with.

    We break the proof into two cases depending on $x$.

    Case 1: ${\mathbf x \geq \tfrac{1}{8}}$. In this case it suffices to take $(K_1, K_2) \in \SO(2) \times \SO(2)$. Consider the quantities
    \begin{equation*}
        P = \tfrac{1}{2}\tr(MM^\dagger) + \det(M) \qquad \text{and} \qquad Q = \tfrac{1}{2}\tr(MM^\dagger) - \det(M).
    \end{equation*}
    One checks that the chosen forms of $K_1, K_2$ guarantee that $\det(M)$, and hence $P,Q$, are real, which also implies $\sigma_1^2 \sigma_2^2 = \det(MM^\dagger) = \det(M)^2$. Let $\Sigma$ be the region $\{(p,q) \in \RR^2 : p,q \geq 0, \quad p+q-1 \leq \tfrac{1}{4}(p-q)^2\}$:

\begin{center}
    \begin{tikzpicture}[scale=0.8]
        \begin{axis}[xlabel=$\tau$, ylabel=$\delta$, axis lines=none, width=0.6\textwidth, height=0.6\textwidth]
            \addplot[domain=0:2, style=thick, samples=500] {2+x-2*sqrt(2*x)};
            \addplot[domain=0:2, style=thick] {0};
            \draw[thick] (axis cs:0, 0) -- (axis cs:0, 2);
        \end{axis}
    \end{tikzpicture}
\end{center}
This is the union of the images of the simplex $1 \geq \sigma_1 \geq \sigma_2 \geq 0$ under the two transformations $(\sigma_1, \sigma_2) \mapsto (\tfrac{1}{2}(\sigma_1 + \alpha \sigma_2)^2, \tfrac{1}{2}(\sigma_1 - \alpha \sigma_2)^2)$ for $\alpha = \pm 1$. It suffices to show that the image of $F : (K_1, K_2) \mapsto (P,Q)$ contains $\Sigma$. Indeed, suppose $P = \tfrac{1}{2}(\sigma_1 + \alpha \sigma_2)^2$ and $Q = \tfrac{1}{2}(\sigma_1-\alpha \sigma_2)^2$. Then $\tr(MM^\dagger) = \tfrac{1}{2}(P+Q) = \sigma_1^2 + \sigma_2^2$ and $\det(M) = \tfrac{1}{2}(P-Q) = \alpha \sigma_1 \sigma_2$, which would mean $\sigma(M) = (\sigma_1, \sigma_2)$.

It is convenient for computational purposes to use the rational parameterization of the unit circle, i.e.\ to make the substitution $s \leadsto 2\tan^{-1}(s)$, replacing $\cos s + i \sin s$ by $\tfrac{1-s^2 + 2si}{1+s^2}$ for $s \in \RR \cup \{\infty\}$ and likewise for $t$:
\begin{equation*}
    K_1 = \frac{1-s^2 + 2si}{1+s^2} \frac{1}{1+t^2} \begin{bmatrix} 1-t^2 & 2ti \\ 2ti & 1-t^2 \end{bmatrix} = K_2^\dagger,
\end{equation*}
Likewise write $\cos \pi x + i \sin \pi x = \tfrac{1-u^2+2ui}{1+u^2}$. View (the restriction of) $F$ as mapping $(s,t) \mapsto (P,Q)$.

It is possible to explicitly solve the equations $P = p, Q = q$ for $s,t$. Computing in the ring $\QQ(i,u)[s,t]$, one checks that these equations generate the same ideal as 
\begin{equation*}
f(y) = Ay^2 + By + C = 0, \qquad  g(z) = pz^2 + (4q-8)z - (4p - 8q - 16) = 0
\end{equation*}
where $y = s^2 + s^{-2}$, $z = t^2+t^{-2}$, and $A,B,C$ are polynomials in $u$ and linear in $p,q$. We must check that these quadratics have real roots $y,z \geq 2$.

The extremum of $g$ occurs at $2(2-q)/p$, whose minimum value on $\Sigma$ is $2$. The value of $g$ at its extremum is $-32(1-\sigma_1^2)(1-\sigma_2^2)(\sigma_1-\sigma_2)^{-2} \leq 0$. Since $g$ has leading coefficient $p \geq 0$, it must have a root $z \geq 2$. 

As for $f$, its discriminant is $16384u^4 (u^2-1)^4 (u^2+1)^8 (1-\sigma_1^2)(1-\sigma_2^2) \geq 0$, so it has real roots. We also have $f(2) = 16p(u^2+1)^8 \geq 0$. Now consider the quantity
\begin{equation*} 
    G = (-\tfrac{B}{2A} - 2)(A) = -4(u^2+1)^4[ (u^8 - 4u^6 + 22u^4 - 4u^2 + 1)p + 8u^2 (u^2-1)^2(q-2) ].
\end{equation*}
We claim that if $G \geq 0$, then $f$ has a root $y \geq 2$. Indeed, if both factors of $G$ are negative, in particular $A$, then $f(\infty) = -\infty$, so $f$ has a root in $[2,\infty)$. If both factors are positive, then the extremum $-B/2A$ occurs right of $2$. We know $f$ has real roots, so $f(\infty) = \infty$ and $f(-B/2A)$ must have opposite signs, hence there is a root in $[2, -B/2A]$.

Now we prove $G \geq 0$ on $\Sigma$. Since $G$ is linear in $p,q$, its extrema on $\Sigma$ occur either on the curved boundary $p+q-1 = \tfrac{1}{4}(p-q)^2$ (i.e. $\sigma_1 = 1$) or else at the vertex $(p,q) = (0,0)$. We have $G(0,0) = 64u(u^2+1)^4(u^2-1)^2 \geq 0$, while
\begin{equation*}
G(\sigma_1 = 1) = 2(1-\sigma_2)(u^2+1)^4[(u^2+1)^4 \sigma_2 - (u^8 - 28u^6 + 70u^4 - 28u^2 + 1)].
\end{equation*}
The minimum of the last factor is $-(u^8 - 28u^6 + 70u^4 - 28u^2 + 1) = -\cos(4\pi x)\sec^8(\pi x/2)$, which is nonnegative so long as $x \in [1/8,1/4]$.

Case 2: $\mathbf{x \leq \tfrac{1}{8}}$. In this case, it suffices to take $K_1$ of the form $e^{is} \left[ \begin{smallmatrix} \cos t & i \sin t \\ i \sin t & \cos t \end{smallmatrix} \right]$ and $K_2 = K_1^\dagger$. Parameterizing the unit circle rationally as in the last case, let
\begin{equation*}
K_1 = \frac{1}{1+s^2} \begin{bmatrix} 1-s^2 & -2s \\ 2s & 1-s^2 \end{bmatrix}, \quad K_2 = \frac{1}{1+t^2} \begin{bmatrix} 1-t^2 & -2t \\ 2t & 1-t^2 \end{bmatrix},
\end{equation*}
and $e^{ix} = \tfrac{1-u^2+2ui}{1+u^2}$. Define $\Sigma$ and $F : (s,t) \mapsto (P,Q)$ as before---except now restrict the domain of $F$ to be $[0,\infty]^2$. We must again show that $F$ is surjective.

First, we claim $\im(F)$ contains a point $z_0$ of the interior of $\Sigma$. This is easy to see: for instance, $F(1,1) = (0,0)$, and one checks that $PQ$ is not identically zero as a rational function unless $x = 1/4$, so $F$ must map a point near $(1,1)$ to a point near $(0,0)$ but not on either axis.

Next, we claim that if $C$ is the set of critical points of $F$, then $F(C) \subseteq \partial \Sigma$.  Up to scalar factors, the Jacobian of $F$ is 
\begin{equation*}
[st(u^2-1)^2 + 4u^2][4u^2 st + (u^2-1)^2](s^2t^2-1)(s-t)(s+t).
\end{equation*}
The first two factors divide $\tfrac{1}{4}(P-Q)^2 - (P+Q-1)$, while $s^2t^2-1$ divides $Q$ and $s-t$ divides $P$. Since we have restricted the domain of $F$ to $[0,\infty]^2$, the remaining factor $s+t$ is nonzero except if $s = t = 0$, in which case $P = 0$. 

Now suppose $z \in \Sigma$. Choose a curve $\gamma : [0,1] \to \Sigma$ connecting $z_0 \in \operatorname{int}(\Sigma)$ to $z$ which lies in $\operatorname{int}(\Sigma)$ except possibly for the endpoint $z = \gamma(1)$. If $z \notin \im(F)$, then there is some maximal $0 < \tau < 1$ with $\gamma(\tau) \in \im(F)$. The point $\gamma(\tau)$ must lie in the boundary $\partial \im(F)$, since otherwise the curve $\gamma$ could be continued a little bit further in $\im(F)$. Then $\gamma(\tau)$ is a critical value of $F$ by the inverse function theorem, which implies $\gamma(\tau) \in \partial \Sigma$ by the previous paragraph. But this contradicts the choice of $\gamma$.

\end{proof}

\begin{theorem} $\B(\SU(2n),\operatorname{S}(\U(n) \times \U(n)))$ is the set of $(\tfrac{1}{2} \geq x_1 \geq \cdots \geq x_n \geq 0)$ with $\tfrac{1}{2}-x_i = x_{n-i+1}$ for all $i$. Thus, the following are equivalent. 
\begin{enumerate}[(a)]
    \item $\K U \K V \K = \SU(2n)$ where $\K = \operatorname{S}(\U(n) \times \U(n))$.
    \item Letting $\theta$ denote conjugation by $\diag(I_n, -I_n)$, both $U \theta(U)^{-1}$ and $V \theta(V)^{-1}$ have the same eigenvalues $e^{\pm \pi ix_1}, \ldots, e^{\pm \pi ix_n}$ where $\tfrac{1}{2} \geq x_1 \geq \cdots \geq x_n \geq 0$ and $x_i + x_{n-i+1} = \tfrac{1}{2}$ for all $i$.
    \item The upper-left $n \times n$ corners of $U, V$ have the same singular values $\sigma_1 \geq \cdots \geq \sigma_n$, which satisfy $\sigma_i^2 + \sigma_{n-i+1}^2 = 1$ for all $i$.
\end{enumerate}
\end{theorem}

\begin{proof}
    Parts (b) and (c) are equivalent by Proposition~\ref{prop:SVD}. Corollary~\ref{cor:SU-block-necc} shows $\B(\SU(2n),\operatorname{S}(\U(n) \times \U(n))) \subseteq \{{\bf x} \in \overline{\alc} : x_i + x_{n-i+1} = \tfrac{1}{2}\}$, so we must show the reverse containment. Take ${\bf x} \in \overline{\alc}$ satisfying $x_i+x_{n-i+1} = \tfrac{1}{2}$.  Let $V = \left[ \begin{smallmatrix} \cos D & i \sin D \\ i \sin D & \cos D \end{smallmatrix}\right]$ where $D$ is diagonal with diagonal entries $\pi x_1, \ldots, \pi x_n$. As in the proof of Lemma~\ref{lem:SU4-suff}, we must show that $M = \cos(D)K_1 \cos(D) - \sin(D) K_2 \sin(D)$ can have any possible list of singular values $\sigma(M)$ in $[0,1]$ with an appropriate choice of $K_1,K_2 \in \U(n)$.

    The singular values $\sigma(M)$ are invariant under multiplying on either side by a permutation matrix, so we can safely rearrange the diagonal of $D$ to the order $\pi x_1, \pi x_n, \pi x_2, \pi x_{n-1}, \ldots$. Now let $H$ be the block-diagonal subgroup
    \begin{equation*}
        \begin{cases}
            \U(2)^{\times n/2} & \text{$n$ even}\\
            \U(2)^{\times (n-1)/2} \times \U(1) & \text{$n$ odd}
        \end{cases}
    \end{equation*}
    in $\U(n)$. If we choose $K_1, K_2 \in H$, then evidently $M$ also has the same block-diagonal structure. By Lemmas~\ref{lem:SU2-suff} and \ref{lem:SU4-suff}, we can choose $K_1, K_2 \in H$ making the first block in $M$ have any singular values $\sigma_1, \sigma_2$, the second block have any singular values $\sigma_3, \sigma_4$, and so on.
\end{proof}

We close this section with an example in which Theorem~\ref{thm:main-1} fails to be a sufficient condition for $\G = \K x \K y \K$. Take $\G = \Sp(n)$ and
\begin{equation*}
\K = \left\{ \begin{bmatrix} X & 0 \\ 0 & \overline{X} \end{bmatrix} : X \in \U(n) \right\} \simeq \U(n).
\end{equation*}
This is very similar to the type AIII case $\SU(2n) \supseteq \S(\U(n) \times \U(n))$, and we can again take 
\begin{itemize}
\item $\a \simeq \RR^n$ the matrices $\left[ \begin{smallmatrix} 0 & iD \\ iD & 0 \end{smallmatrix} \right]$ with $D$ real diagonal 
\item fundamental alcove $\alc = \{{\bf x} \in \RR^n : \tfrac{1}{2} > x_1 > \cdots > x_n > 0\}$.
\item finite Weyl group $B_n$
\item for $g \in \G$, if $a(g) = {\bf x} \in \overline{\alc}$ then the upper-left $n \times n$ corner of $g$ has singular values $\cos(\pi x_n), \ldots, \cos(\pi x_1)$.
\end{itemize}
Thus, Theorem~\ref{thm:main-2} asserts that 
\begin{equation*}
\B(\Sp(2), \U(2)) \subseteq \{(\tfrac{1}{2}-x,x) : x \in [0,\tfrac{1}{4}]\},
\end{equation*}
but unlike in Lemma~\ref{lem:SU4-suff}, this is \emph{not} an equality. For instance, take $x = 0$, so $D = \diag(\pi/2, 0)$ and
\begin{equation*}
V = \begin{bmatrix} \cos D & i\sin D \\ i \sin D & \cos D \end{bmatrix} = \begin{bmatrix}
0 & 0 & i & 0\\
0 & 1 & 0 & 0\\
i & 0 & 0 & 0\\
0 & 0 & 0 & 1
\end{bmatrix}.
\end{equation*}
Then $\K V \K V^{-1} \K = \G$ would require that the upper-left $2 \times 2$ corner of $V \left[\begin{smallmatrix} X & 0 \\ 0 & \overline{X}\end{smallmatrix}\right] V^{-1}$ can have any possible pair of singular values in $[0,1]$ for an appropriate $X = e^{i\alpha}\left[\begin{smallmatrix} z & -\overline{w} \\ w & \overline{z} \end{smallmatrix}\right] \in \U(2)$. But this upper-left corner is 
\begin{equation*}
\cos(D)X\cos(D) + \sin(D)\overline{X}\sin(D) = \left[\begin{smallmatrix} e^{-i\alpha}\overline{z} & 0 \\ 0 & e^{i\alpha}z \end{smallmatrix}\right],
\end{equation*}
with singular values $(|z|,|z|)$. With some extra work one can show that in fact $\B(\Sp(2), \U(2)) = \{(\tfrac{1}{2}-x,x) : x \in [\tfrac{1}{8},\tfrac{1}{4}]\}$ (note that half of this work is the ${\bf x} \leq \tfrac{1}{8}$ case in the proof of Lemma~\ref{lem:SU4-suff}).

\section{Applications to quantum gate decompositions} \label{sec:gates}
An \emph{$n$-qubit gate} is an element of $\U(2^n)$. In fact, two gates are considered the same if they differ by a phase factor, so working in $\PSU(2^n)$ would be more accurate, but we will not worry about this. If $V \in \U(2^m)$ and $W \in \U(2^n)$ then one can form the $(m+n)$-qubit gate $V \otimes W$, which does not mix the states of qubits $1, \ldots, m$ and qubits $m+1, \ldots, m+n$. At the extreme end is the subgroup of \emph{single-qubit gates} $\U(2)^{\otimes n} \subseteq \U(2^n)$---note that this terminology is somewhat ambiguous since it could also refer simply to elements of $\U(2)$.

An important problem in quantum computing is \emph{gate decomposition}: fix a small set of ``nice'' gates $S \subseteq \U(2^n)$, and attempt to factor arbitrary $n$-qubit gates as a product of elements of $S$ plus gates acting only within smaller groups of qubits, i.e. elements of $\U(2^{a_1}) \otimes \cdots \otimes \U(2^{a_m})$ where $a_1 + \cdots + a_m = n$. 

For example, recall the \emph{CNOT} (controlled-not) gate $C$ from the introduction. More generally, we can let a CNOT act on 2 qubits $i$ and $j$ out of $n$ total, to get an $n$-qubit gate. That is, if we take $(\CC^2)^{\otimes n}$ to have basis vectors $\ket{b_1 \cdots b_n}$ over binary words $b_1\cdots b_n$, then a CNOT with control qubit $i$ and target qubit $j$ acts as
\begin{equation*}
\ket{{\mathbf b}} \mapsto \begin{cases} \ket{{\mathbf b}} & \text{if $b_i = 0$}\\
    \ket{b_1 \cdots b_{j-1} \operatorname{NOT}(b_j) b_{j+1} \cdots b_n} & \text{if $b_i = 1$}
\end{cases}
\end{equation*}
Various algorithms have been developed to decompose arbitrary $n$-qubit gates into a product of CNOTs and single-qubit gates. Shende, Markov, and Bullock \cite{shende-markov-bullock} showed using a dimension-counting argument that at least $(4^n-3n-1)/4$ CNOTs are required in any such decomposition holding for all $n$-qubit gates. The current most efficient algorithm seems to be due to Krol and Al-Ars \cite{block-ZXZ}, requiring $\leq \tfrac{22}{48} 4^n - \tfrac{3}{2} 2^n + \tfrac{5}{3}$ CNOTs.

The 2-qubit case has an interesting property that makes it more tractable, if still nontrivial. Consider the two standard labelings of the $A_n$ and $D_n$ Dynkin diagrams:
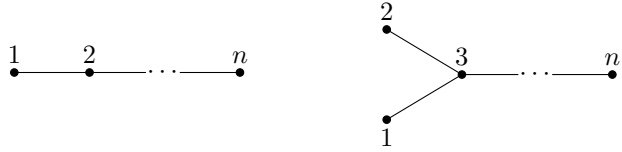
\begin{center}
\raisebox{9mm}{\begin{tikzpicture}
\filldraw[color=black] (0,0) circle (0.5mm);
\filldraw[color=black] (1,0) circle (0.5mm);
\filldraw[color=black] (3,0) circle (0.5mm);
\draw[above] node at (0,0) {$1$};
\draw[above] node at (1,0) {$2$};
\draw[above] node at (3,0) {$n$};
\draw node at (2,0) {$\cdots$};

\draw (0,0) -- (1,0) -- (1.75,0);
\draw (2.2,0) -- (3,0);
\end{tikzpicture}} \qquad \qquad
\begin{tikzpicture}
\filldraw[color=black] (0,-0.6) circle (0.5mm);
\filldraw[color=black] (0,0.6) circle (0.5mm);
\filldraw[color=black] (1,0) circle (0.5mm);
\filldraw[color=black] (3,0) circle (0.5mm);
\draw[below] node at (0,-0.6) {$1$};
\draw[above] node at (0,0.6) {$2$};
\draw[above] node at (1,0) {$3$};
\draw[above] node at (3,0) {$n$};
\draw node at (2,0) {$\cdots$};

\draw (0,-0.6) -- (1,0);
\draw (0,0.6) -- (1,0) -- (1.75,0);
\draw (2.2,0) -- (3,0);
\end{tikzpicture}
\end{center}
According to these labelings, $D_2 = A_1 \times A_1$, and indeed there is an exceptional isomorphism of Lie algebras $\su(2) \oplus \su(2) \simeq \so(4)$. At the group level this manifests as the unexpected equality $\mathcal{Q}^\dagger \SU(2)^{\otimes 2} \mathcal{Q} = \SO(4)$, where $\mathcal{Q} = \frac{1}{2}\left[ \begin{smallmatrix}
1 & 1 & i & i\\
1 & -1 & i & -i\\
-1 & 1 & i & -i\\
1 & 1 & -i & -i
\end{smallmatrix}\right]$ is a so-called \emph{Bell matrix} or \emph{magic matrix}.  Thus, $\SU(2)^{\otimes 2}$ is a Cartan subgroup of $\SU(4)$, the fixed-point set of the involution $\theta(U) = \overline{\mathcal{Q}^\dagger U \mathcal{Q}}$.

\begin{proposition} \label{prop:min-2q}
Let $\K = \SU(2) \otimes \SU(2)$ and $U,V \in \SU(4)$. Then $\K U \K V \K = \SU(4)$ if and only if $U,V$ are both equivalent to the \emph{Berkeley gate}
\begin{equation*}
B = \begin{bmatrix}
        \cos(\pi/8) & 0 & 0 & i\sin(\pi/8)\\
        0  & \cos(3/\pi/8) & i\sin(3\pi/8) & 0\\
        0  & i\sin(3/\pi/8) & \cos(3\pi/8) & 0\\
        i\sin(\pi/8) & 0 & 0 & \cos(\pi/8)
\end{bmatrix}
\end{equation*}
up to multiplication by single-qubit gates.
\end{proposition}

\begin{proof}
According to Theorem~\ref{thm:main-2}, $\K U \K V \K = \SU(4)$ holds if and only if $M = (\Q^\dagger U \Q)(\Q^\dagger U \Q)^T$ satisfies $M^4 = -I$, and likewise with $U$ replaced by $V$. A direct calculation shows that this holds for $U = B$, and we know this condition uniquely characterizes the $\K$-double coset of $U$ by Lemma~\ref{lem:KAK-conj}.
\end{proof}

The equation $\K B \K B \K = \SU(4)$ is not new \cite{berkeley-gate}, but Proposition~\ref{prop:min-2q} shows that the Berkeley gate is essentially unique with this minimal decomposition property, answering a question from \cite{peterson-crooks-smith}.

Next we turn to an application of Theorem~\ref{thm:main-2} in type AIII. Suppose $F,G \in \U(2)$ are 1-qubit gates with $|F_{11}| = |G_{11}| = 1/\sqrt{2}$. Then the upper-left $2^{n-1} \times 2^{n-1}$ corner of $F \otimes I_{2^{n-1}}$ is $F_{11}I_{2^{n-1}}$, with singular values all equal to $1/\sqrt{2}$, and likewise for $G$. Therefore Theorem~\ref{thm:main-2} (case AIII) says any $n$-qubit gate can be decomposed as 
\begin{align} \label{eq:block-ZXZ}
&\begin{bmatrix} P & 0 \\ 0 & A \end{bmatrix} (F \otimes I_{2^{n-1}})
\begin{bmatrix} Q & 0 \\ 0 & B \end{bmatrix} (G \otimes I_{2^{n-1}})
\begin{bmatrix} R & 0 \\ 0 & C \end{bmatrix} \nonumber \\
=\, &\begin{bmatrix} I & 0 \\ 0 & A' \end{bmatrix} (F \otimes I_{2^{n-1}})
\begin{bmatrix} I & 0 \\ 0 & B' \end{bmatrix} (G \otimes I_{2^{n-1}})
\begin{bmatrix} S & 0 \\ 0 & C' \end{bmatrix},
\end{align}
where we have simplified using the fact any matrix $\diag(M,M) = I_2 \otimes M$ commutes with any $N \otimes I_{2^{n-1}}$. A natural choice is to take $F = G$ to be the \emph{Hadamard gate} $H = \tfrac{1}{\sqrt 2} \left[ \begin{smallmatrix} 1 & 1 \\ 1 & -1 \end{smallmatrix} \right]$, in which case \eqref{eq:block-ZXZ} is the \emph{block-ZXZ decomposition} from \cite{block-ZXZ-orig}. This factorization was used by Krol and al-Ars \cite{block-ZXZ} to find a new general gate decomposition involving fewer CNOTs than any previously known. 
 
We note the following theorem of Gupta and Hare which may lead to weaker but potentially still useful decompositions of elements of $\U(n)$.
\begin{theorem}[\cite{gupta-hare}, Theorem 3.1]
Suppose $x,y \in \G$ and $a(x), a(y)$ are regular elements of $\G$. Then $\K x \K y \K$ has nonempty interior.
\end{theorem}
Here, an element $g \in \G$ is regular if its centralizer has minimal possible dimension (equal to $\dim \h$). For example, $g \in \G = \U(n)$ is regular exactly if it has distinct eigenvalues. If $\K x \K y \K$ has nonempty interior, then it has positive Haar measure, so such sets $\K x \K y \K$ could be used to construct decompositions which may not work for all elements of $\U(n)$, but at least work with positive probability. As the set of regular elements is dense in $\G$, such decompositions are much easier to come by than those coming from an exact Cartan decomposition.

\section*{Acknowledgements}
I thank Jim van Meter for help navigating the literature on Cartan decompositions and quantum gate decompositions.

\bibliographystyle{plain}
\bibliography{Q}

\end{document}